\documentclass[preprint,10pt,3p]{elsarticle}
\usepackage{graphics}

\usepackage{epstopdf}

\usepackage{amssymb}
\usepackage{amsthm}
\usepackage{amscd}
\usepackage{amsmath}
\usepackage{amsfonts}
\usepackage{mathtools}
\usepackage{geometry}
\usepackage{longtable}
\usepackage{pdfsync}
\usepackage{boldline,multirow}
\usepackage{xcolor}
\usepackage{float}
\usepackage{booktabs, multirow, caption}
\usepackage[labelfont=bf,textfont={sl,bf},lofdepth,lotdepth]{subfig}

\usepackage[active]{srcltx}

\DeclareCaptionListOfFormat{colon}{#2:}

 \newcommand{\alp}{\alpha}

\theoremstyle{definition}
\newtheorem{definition}{Definition}[section]
\newtheorem{example}{Example}
\newtheorem{lem}{Lemma}[section]
\newtheorem{thm}{Theorem}[section]
 \newtheorem{rem}{Remark}[section]
\newtheorem{alg}[]{Algorithm}

\usepackage{mathtools}

\begin{document}

\begin{frontmatter}


\title{Efficient and fast predictor-corrector method for solving nonlinear fractional differential equations with non-singular kernel}
\author[label1]{Seyeon Lee}
\author[label1]{Junseo Lee}

\author[label2]{Hyunju Kim}
\address[label2]{Department of Mathematics, North Greenville University, Tigerville, SC USA 29688}

\author[label1]{Bongsoo Jang\corref{cor1}}
\address[label1]{Department of Mathematical Sciences, Ulsan National Institute of Science and
Technology(UNIST), Ulsan 44919, Republic of
Korea}

\cortext[cor1]{Corresponding author. Tel:+82 52 217 3136}

\ead{bsjang@unist.ac.kr}


\date{\today}

\begin{abstract}
Efficient and fast predictor-corrector methods are proposed to deal with nonlinear Caputo-Fabrizio fractional differential equations, where Caputo-Fabrizio operator is a new proposed fractional derivative with a smooth kernel.
The proposed methods achieve a uniform accuracy order with the second-order scheme for linear interpolation and the third-order scheme for quadratic interpolation. The convergence analysis is proved by using the discrete Gronwall's inequality.
Furthermore, applying the recurrence relation of the memory term, it reduces CPU time executed the proposed methods. The proposed fast algorithm requires approximately $\mathcal{O}{(N)}$ arithmetic operations while $\mathcal{O}{(N^2)}$ is required in case of the regular predictor-corrector schemes, where $N$ is the total number of time step.


The following numerical examples demonstrate the accuracy of the proposed methods as well as the efficiency; nonlinear fractional differential equations, time-fraction sub-diffusion, and time-fractional advection-diffusion equation. The theoretical convergence rates are also verified by numerical experiments.

\end{abstract}


\end{frontmatter}


\section{Introduction}\label{sec_intro}

Nowadays, fractional differential equations have received the attention of many researchers in the diverse area of applied sciences, technology and engineering because it can reflect memory effect. For example, fractional differential equations can be applied to anomalous diffusion transport \cite{GHPS, KRS, Z},  porous media \cite{BWM1}, viscoelastic materials \cite{BC, Ma}. We also find applications in chaos, chemistry, finance, geo-hydrology, nonlinear dynamics \cite{Ma, petravs2011fractional, baleanu2011fractional, P, atangana2017fractional, SAM}.

 There are popular definitions of fractional derivative and integration which are introduced in \cite{petravs2011fractional, atangana2017fractional, AG, CF1} such as Gr\"{u}nwald-Letnikov, Riemann-Liouville, Caputo, etc.  Recently, several new fractional derivatives with non-singular kernels  have been suggested such as Atangana - Gmez, Atangana-Baleanu and Caputo-Fabrizio\cite{AG, CF1, atangana2017}. 
In this work, we consider the following nonlinear differential equation with fractional order $\alp\in (0,1)$ in the Caputo-Fabrizio sense ($D^{\alp}_a \equiv ~^{CF}_a D^{\alp}_t$) for $y(t) \in H^1[0,T]$,
\begin{align}\label{intro_mod}
\begin{cases}
&D^{\alp}_ay(t)=f(t,y(t)),\quad t\in[0,T],\\
&y(0)=y_0.
\end{cases}
\end{align}

The Caputo-Fabrizio fractional derivative is defined as follows:
\begin{definition}\cite{CF1} 
	Let $y \in H^1(a,b), \alp \in [0,1]$, then the Caputo-Fabrizio fractional derivative of order $\alp$ is given by
	\begin{equation}\label{def_cf}
		~^{CF}_a D^{\alp}_t y(t)=\frac{M{(\alp)}}{1-\alp} \int_a^t y^{'}(s) \exp\left[ -\frac{\alp}{1-\alp}(t-s)\right]ds,
	\end{equation}
	where $M(\alp)$ is a normalization function such that $M(0)=M(1)=0$. However, if the function does not belongs to $H^1(a,b)$, then the derivative can be defined as
	\begin{equation*}
		~^{CF}_a D^{\alp}_t y(t)=\frac{ M{(\alp)}}{1-\alp} \int_a^t \Bigl{(}y(t)-y(s)\Bigr{)} \exp\left[ -\frac{\alp}{1-\alp}(t-s)\right]ds.
	\end{equation*}
\end{definition}

In papers \cite{DFF02, JANG1}, authors proposed predictor-corrector schema equipped with the approximation of $f(t,y(t))$ by the linear or quadratic interpolation to approximate the solution of the Volterra integral equation which is equivalent to that of the fractional order differential equation \eqref{intro_mod} with a singular fractional derivative such as the Caputo fractional derivative. And many authors developed numerical methods for solving Volterra integral equations to evade dealing with singular kernel in the fractional differential equation \eqref{intro_mod} \cite{Brunner-linear, Brunner-nonlinear, Chen1, Gu, Gu-Chen, Huang1, Shen1, Wei1, Xianjuan1}.
Comparably, for the Caputo-Fabrizio fractional derivative, it can be converted into the integral formulation as 
\begin{equation} \label{int_form}
	y(t)=y_0+\frac{1-\alp}{M(\alp)}f(t,y(t))+\frac{\alp}{M(\alp)} \int_0^t f(s,y(s)) ds,
\end{equation}
by taking the Laplace transform to the model problem (\ref{intro_mod}).
Authors in papers \cite{atangana2017, KOCA2018278} introduced numerical methods approximating $y(t)$ in (\ref{int_form}) by using a linear interpolation and quadratic interpolation of $f(t,y(t))$. The paper \cite{atangana2016numerical} applied the first-order approximation of $f'(s)$ to the Caputo-Fabrizio fractional derivative in $(\ref{def_cf})$.

In this work, we introduce a new transformation of the model problem. Using the integration by parts, the model problem \eqref{intro_mod} is transformed involving the non-singular kernel as follows;
\begin{equation}\label{model}	
	y(t)=g(t,y(t))+\beta \int_0^t y(s) \exp\left[ -\beta(t-s)\right]ds,
\end{equation}
where 
\begin{equation}\label{gx_lin}
	g(t,y(t))=\frac{1-\alp}{M(\alp)} f(t,y(t))+y_0 \exp\left[ -\beta t \right], \quad \beta=\frac{\alp}{1-\alp}.
\end{equation}

The main results are described as
\begin{enumerate}
	\item New predictor-corrector schemes of (\ref{model}) by using the linear and quadratic interpolation of $y(s)$ in the integral term are proposed. Also we obtain the global convergence analysis for the proposed methods that is the secod-order scheme with linear interpolation and the third-order scheme with quadratic interpolation for any fractional order $0<\alpha<1$, respectively. 
	\item The computational cost of a numerical method for solving the model problem is high in general due to the non-local property in (\ref{model}). In order to overcome this drawback, we propose a new class of predictor-corrector scheme by constructing a recurrence relation of the memory term with an exponential kernel. The memory term is updated by adding a histry term into a local term recursively. This fast algorithm reduces the arithmetic operations required from $\mathcal{O}{(N^2)}$ to $\mathcal{O}{(N)}$ approximately, where $N$ is the total number of time step. 
	\item We extend the proposed method to solve the fractional partial differential equations such as the time-fractional sub-diffusion and time-fractional advection problems. All numerical results support the theoretical results in convergence analysis and show the reduction of computational cost. 
\end{enumerate}
For the numerical method, let us discretize the grid to be
\begin{align}\label{intro_grid}
\Phi_N:=\{t_j:0=t_0<\ldots<t_j<\ldots<t_{n}<t_{n+1}<\ldots<t_{N+1}=T\}.
\end{align}
We assume that the grid is uniform for simplicity, i.e., $h=t_{j+1}-t_j$, $\forall j=0,\ldots,N-1$. 
Then, Eq. (\ref{model}) can be rewritten at $t=t_{n+1}$ as follows 
\begin{align}\label{lin_sol}
	y(t_{n+1})=g(t_{n+1},y(t_{n+1}))+\beta \int_{t_0}^{t_{n+1}} y(s) \exp\left[ -\beta(t_{n+1}-s)\right]ds.
\end{align}

Our paper is organized as follows. In Section \ref{sec_lin}, we describe a new predictor-corrector scheme with linear interpolation and its error analysis. The second-order of accuracy is achieved by the linear interpolation for any $0<\alp<1$. In similar to Section \ref{sec_lin}, we describe the improved predictor-corrector scheme with quadratic interpolation and the third-order of accuracy in Section \ref{sec_quad}. In Section \ref{sec_fast}, we provide a fast algorithm for computing the memory part of (\ref{lin_sol}) which reduces the number of arithmetic operations required from $O(N^2)$ to $O(N)$ by using a recurrence relation. In Section \ref{sec_num}, we demonstrate the performance of the proposed methods by solving numerical tests. The error estimates shown in Section \ref{sec_lin} and \ref{sec_quad} are verified through the numerical examples. Additionally, The proposed methods are applied to the fractional partial differential equations; the time-fractional sub-diffusion and advection diffusion equation. Our conclusions are depicted in the last section.

\section{Second-order Predictor-Corrector Scheme with Linear Interpolation}\label{sec_lin}


Let $y_{j}$ be the approximation of $y(t_{j})$, $j=0,\ldots,N+1$. 
On each interval $I_j=[t_j,t_{j+1}]$, we interpolate $y(t)$ by a linear Lagrange polynomial then we have the following expression:
\begin{equation*}\label{lin_interp}
 y(t)= L^1_j y(t)+R_j(y(t)),
\end{equation*}
where
\begin{equation*}
	L^1_j y(t)=\frac{t_{j+1}-t}{h} y(t_j) +\frac{t-t_j}{h} y(t_{j+1}), \quad R_j(y(t))=\frac{y''(\xi_j)}{2}(t-t_j)(t-t_{j+1}), \xi_j \in (t_j, t_{j+1}), 
\end{equation*}
From (\ref{lin_sol}) we have
\begin{align} \label{lin_sol2}
	y(t_{n+1})&=g(t_{n+1}, y(t_{n+1}))
		+\beta \sum_{j=0}^{n} \int_{t_{j}}^{t_{j+1}} \left( L^1_{j}y(s)+ R_j(y(s)) \right)\exp\left[ -\beta(t_{n+1}-s)\right]ds,\\
		&=g(t_{n+1}, y(t_{n+1}))
		+\beta \sum_{j=0}^{n}(B^{1,j}_{n+1} y(t_{j})+B^{2,j}_{n+1} y(t_{j+1}))
		+T_{n+1}, \nonumber
\end{align}
where
\begin{align*}
	& B^{1,j}_{n+1}=\frac{1}{h}\int_{t_{j}}^{t_{j+1}}(t_{j+1}-s)\exp\left[ -\beta(t_{n+1}-s)\right]ds,\\
       & B^{2,j}_{n+1}=\frac{1}{h}\int_{t_{j}}^{t_{j+1}}(s-t_{j})\exp\left[ -\beta(t_{n+1}-s)\right]ds,\\
       & T_{n+1}=\beta \sum_{j=0}^{n} \int_{t_{j}}^{t_{j+1}}  R_{j}(y(s))\exp\left[ -\beta(t_{n+1}-s)\right]ds.\end{align*}
Then $y(t_{n+1})$ can be evaluated by solving the following equations:
\begin{align}\label{exactn}
	&\left[ 1-\beta B^{2,n}_{n+1}\right] y(t_{n+1})\\ \nonumber
	&=g(t_{n+1}, y(t_{n+1}))+\beta \left[ \left\{\sum_{j=0}^{n-1}(B^{1,j}_{n+1}y(t_{j})+B^{2,j}_{n+1} y(t_{j+1}))\right\}
		+B^{1,n}_{n+1}y(t_{n})\right]
		+T_{n+1}.
\end{align}
Therefore, the approximate value $y_{n+1}$ of $y(t_{n+1})$ can be obtained by solving the following scheme
\begin{equation*}
	\left[ 1-\beta B^{2,n}_{n+1}\right] y_{n+1}=g(t_{n+1}, y_{n+1})
		+\beta \left[ \left\{\sum_{j=0}^{n-1}(B^{1,j}_{n+1} y_{j}+B^{2,j}_{n+1} y_{j+1})\right\}+B^{1,n}_{n+1}y_{n}\right].
\end{equation*}
Using the linear interpolation of $f(t,y(t))$ with grid points $t_{n-1}$ and $t_{n}$, the value of $f(t,y(t))$ at $t=t_{n+1}$ can be evaluated by 
\begin{equation} \label{lin_g}
		 f(t_{n+1}, y(t_{n+1})) = -f(t_{n-1},y(t_{n-1}))+2 f(t_{n}, y(t_{n}))+f''(\eta,y(\eta)) h^2, \quad \eta \in (t_{n-1},t_{n}).
\end{equation}
Using (\ref{gx_lin}) and (\ref{lin_g}) the approximation of $g(t,y(t))$ at $t=t_{n+1}$ can be defined by 
\begin{equation}\label{gnp1}
	g_{n+1}=\frac{1-\alp}{M(\alp)}  (-f_{n-1}+2 f_n)+y_0 \exp\left[ -\beta t_{n+1}\right].
\end{equation}
Thus we propose the predictor-corrector scheme as follows
\begin{equation}\label{pecelin}
	\begin{split} 
	\left[ 1-\beta B^{2,n}_{n+1}\right] y_{n+1}&=g(t_{n+1}, y^p_{n+1})
		+\beta \left[ \left\{\sum_{j=0}^{n-1}(B^{1,j}_{n+1}y_{j}+B^{2,j}_{n+1} y_{j+1})\right\}+B^{1,n}_{n+1}y_{n}\right] \\  
	\left[ 1-\beta B^{2,n}_{n+1}\right] y^p_{n+1}&= g_{n+1}
		+\beta \left[ \left\{\sum_{j=0}^{n-1}(B^{1,j}_{n+1}y_{j}+B^{2,j}_{n+1} y_{j+1})\right\}+B^{1,n}_{n+1}y_{n}\right].
	\end{split}	
\end{equation}

\subsection{Error Analysis for Linear Interpolation }

From herein, we denote $C$ a generic constant which is independent of all grid parameters and may change case by case.

We need the following lemmas.

\begin{lem}\label{lem_interp_err}(\emph{Interpolation Errors})
Let $f\in\mathcal{C}^{n+1}[a,b]$ and $p_n\in\mathbb{P}_n[a,b]$ interpolate the function $f$ at the grid $\Phi_n$ in (\ref{intro_grid}) with $a=t_0$ and $b=t_n$, then there exists $\xi\in(a,b)$ such that, for any $t\in[a,b]$,
\begin{align*}
f(t)-p_n(t)=\frac{f^{(n+1)}(\xi)}{(n+1)!}\prod_{j=0}^n(t-t_j).
\end{align*}
\end{lem}

\begin{lem}\label{lem_Gronwall}(\emph{Discrete Gronwall's Inequality}) (\cite{DK}, \cite{CX})
Let $\{a_n\}_{n=0}^N$, $\{b_n\}_{n=0}^N$ be non-negative sequences with $b_n$'s monotonic increasing, and satisfy that
\begin{align*}
a_n\le b_n+Mh^{\gamma}\sum_{j=0}^{n-1}(n-j)^{\gamma-1}a_j,\quad 0\le n\le N,
\end{align*}
where $M>0$ is bounded and independent of $h$, and $0<\gamma\le1$.
Then,
\begin{align*}
a_n\le b_nE_{\gamma}(M\Gamma(\gamma)(nh)^{\gamma}),
\end{align*}
where $E_{\gamma}(\cdot)$ is the Mittag-Leffler function (\cite{D}). It is noted that when $\gamma=1$, the Mittag-Leffler function becomes an exponential one.
\end{lem}

\begin{lem}\label{lem_1}

\begin{equation*}
1-\beta B^{2,n}_{n+1}=1-\frac{1}{2} \beta h+\mathcal{O}(h^2).
\end{equation*}
\end{lem}

\begin{proof}
A simple calculation of $B^{2,n}_{n+1}$ and the Taylor expansion of exponential function gives the following
\begin{align*}
	1-\beta B^{2,n}_{n+1}=\frac{1}{\beta h} \left[ 1-\exp (-\beta h)\right]
	&=\frac{1}{\beta h} \left[ 1-\left(1 -\beta h +\frac{1}{2} \left(\beta h\right)^2-\frac{1}{3!}(\beta h)^3+\cdots \right) \right]\\
	&=1-\frac{1}{2}(\beta h)+\frac{1}{3!}(\beta h)^2+\cdots.
\end{align*}
We then deduce the result. 
\end{proof}
\begin{lem}\label{lem_2}
For $j=0,\cdots,n$ and $i=1,2$,
\begin{equation*}
| B^{i,j}_{n+1}| \le h. 
\end{equation*}
\end{lem}

\begin{proof}
For any $s \in [t_j,t_{j+1}]$, since
$$ \exp\left[ -\beta(t_{n+1}-s)\right] \le 1,$$
we have
\begin{equation*}
	B^{1,j}_{n+1}=\frac{1}{h}\int_{t_{j}}^{t_{j+1}}(t_{j+1}-s)\exp\left[ -\beta(t_{n+1}-s)\right]ds
	\le \int_{t_{j}}^{t_{j+1}}\exp\left[ -\beta(t_{n+1}-s)\right]ds \le h.
\end{equation*}
\end{proof}

\begin{lem}\label{lem_3}
For $y \in C^2[0, T]$, the truncation error $T_{n+1}$ can be estimated by
\begin{equation*}
 |T_{n+1}|= \mathcal{O}(h^2). 
\end{equation*}
\end{lem}

\begin{proof}

\begin{equation*}
	\begin{split}
	 |T_{n+1}|&=\left|\beta \sum_{j=0}^{n} \int_{t_{j}}^{t_{j+1}}  R_{j}(y(s))\exp\left[ -\beta (t_{n+1}-s) \right]ds \right|\\
	            & \le \beta \frac{||y''||_{\infty}}{2} h^2 \int_{t_{0}}^{t_{n+1}} \exp\left[ -\beta (t_{n+1}-s)\right]ds
	             \le \frac{||y''||_{\infty}}{2} \left(1-\exp\left[ -\beta t_{n+1}\right]\right)   h^2.
	            \end{split}
\end{equation*}
\end{proof}

	


\begin{thm}\label{thm_pre_glo}(Global Error)
Suppose that $f(\cdot, u(\cdot))\in\mathcal{C}^2[0,T]$, and furthermore $f(\cdot, u(\cdot))$ is Lipschitz continuous in the second argument, i.e.,
\begin{align} \label{lipf}
|f(t,u_1)-f(t,u_2)|\le L|u_1-u_2|,\quad\forall u_1,u_2\in\mathbb{R}.
\end{align}
Let $E^P_{n+1}=y(t_{n+1})-y^p_{n+1}$ and $E_j=y(t_j)-y_j$. Then there exist a small $h \ll 1$ such that  global error $|E^P_{n+1}|$ is
\begin{equation}\label{pre_err}
|E^P_{n+1} |\le  \frac{2L(1-\alp)}{|M(\alp)|} (|E_{n-1}|+2|E_n | )+4 \beta h\sum_{j=1}^{n} |E_j|+\mathcal{O}(h^2)
\end{equation}
\end{thm}
\begin{proof}
Subtracting the equation of predictor in (\ref{pecelin}) from (\ref{exactn}) we have
\begin{align}
	\left[ 1-\beta B^{2,n}_{n+1}\right] E^P_{n+1}
	&=E_{g,n+1}+\beta \left[ \left\{\sum_{j=0}^{n-1}(B^{1,j}_{n+1}E_j+B^{2,j}_{n+1} E_{j+1})\right\}+B^{1,n}_{n+1}E_{n}\right]+T_{n+1}\\
	&=E_{g,n+1}+\beta \left[ B^{1,1}_{n+1} E_0+\sum_{j=1}^{n}(B^{2,j-1}_{n+1}+B^{1,j}_{n+1})E_j \right]+T_{n+1},
\end{align}
where $E_{g,n+1}=g(t_{n+1}, y(t_{n+1}))-g_{n+1}$. \\
By Lemma \ref{lem_1}, we can choose $h \ll 1$ for any $\alp \in (0,1)$ such that
\begin{equation}\label{apph}
	\frac{1}{2} \le 1-\beta B^{2,n}_{n+1}.
\end{equation}
Then, the Lemma \ref{lem_2} and \ref{lem_3} implies
\begin{align}\label{lin_ep}
	|E^P_{n+1}|
	&\le 2|E_{g,n+1}|+4 \beta h\sum_{j=0}^{n} |E_j|+\mathcal{O}(h^2).
\end{align}
Using $(\ref{lin_g})$ and $(\ref{gnp1})$ combined with Lipschitz condition of $f$ in (\ref{lipf}), we have
\begin{equation} \label{egnp1}
	|E_{g,n+1}| \le \frac{1-\alp}{|M(\alp)|} L  (|E_{n-1}|+2|E_n | )+\mathcal{O}(h^2).
\end{equation}
Substituting (\ref{egnp1}) into (\ref{lin_ep}), it completes the proof.
\end{proof}

\begin{thm}\label{thm_glo}(Global Error)
Suppose that $f(\cdot, u(\cdot))\in\mathcal{C}^2[0,T]$, and furthermore $f(\cdot, u(\cdot))$ is Lipschitz continuous in the second argument, i.e.,
\begin{align} 
|f(t,u_1)-f(t,u_2)|\le L|u_1-u_2|,\quad\forall u_1,u_2\in\mathbb{R},
\end{align}
then there exist a small $h \ll 1$ such that  global error $|E_{n+1}|$ is
\begin{equation*}
	 |E_{n+1}|
	\le \left\{ C_1( |E_{n-1}|+2|E_n|)+C_2h^2\right\}
	     \exp \left[ C_3  T \right],
\end{equation*}
where $C_i, i=1,2,3$ are independent constants of all grid parameters.
\end{thm}
\begin{proof}
Subtracting the equation of predictor in (\ref{pecelin}) from (\ref{exactn}) and using the Lipschitz condition of $f$ in (\ref{lipf}), we have
\begin{equation}
	\left[ 1-\beta B^{2,n}_{n+1}\right] |E_{n+1}|
	\le \frac{L(1-\alp)}{|M(\alp)|} |E^P_{n+1}|+2 \beta h\sum_{j=0}^{n} |E_j|+|T_{n+1}|.
\end{equation}
By choosing $h$ in  (\ref{apph}) and applying the error of $|E^P_{n+1}|$ in (\ref{pre_err}),  we have
\begin{align*}
	 |E_{n+1}|
	&\le \frac{2L(1-\alp)}{|M(\alp)|} |E^P_{n+1}|+4 \beta h\sum_{j=0}^{n} |E_j|+2|T_{n+1}|\\
	&\le \left( \frac{2L(1-\alp)}{M(\alp)} \right)^2( |E_{n-1}|+2|E_n|)
		+4\beta  \left( \frac{2L(1-\alp)}{|M(\alp)|} +1\right) h\sum_{j=0}^n |E_j| +Ch^2,
\end{align*}
where $C$ is a generic constant.
By the Discrete Gronwall's inequality in Lemma \ref{lem_Gronwall} , we deduce 
\begin{equation*}
	 |E_{n+1}|
	\le \left\{ \left( \frac{2L(1-\alp)}{M(\alp)} \right)^2( |E_{n-1}|+2|E_n|)+Ch^2\right\}
	     \exp \left[4\beta \left( \frac{2L(1-\alp)}{|M(\alp)|} +1\right)   T \right].
\end{equation*}
\end{proof}

\begin{thm}\label{thm_err}(Global Error)
With the same assumptions as those of Theorem \ref{thm_glo}, we have
\begin{equation}\label{glb-err-linear}
	 |E_{n+1}| \le C h^2,
\end{equation}
if the starting error $|E_1| \le Ch^2$ is given.
\end{thm}
\begin{proof}
it can be shown by using mathematical induction on $n$ and Theorem \ref{thm_glo}.
\end{proof}
\section{Third-order Predictor-Corrector Scheme with Quadratic Interpolation}
\label{sec_quad}


In this section, we employ a quadratic interpolation of $y(t)$ over each interval $I_j=[t_j,t_{j+1}]$.
For each $I_j, j\ge 1$, we interpolate $y(t)$ by a quadratic Lagrange polynomial
\begin{align*}\label{quad_interp}
y(t)= L^2_j y(t)+R^2_j(t),
\end{align*}
where
\begin{align*}
	L^2_j y(t)&=\sum_{k=j-1}^{j+1} Q^j_k(t) y(t_{k}), \quad Q^j_k(t)=\prod_{\substack{m=j-1\\m\ne k}}^{j+1}\frac{t-t_m}{t_k-t_m},\\
	R^2_j y(t)&=\frac{y'''(\xi_j)}{6}(t-t_{j-1})(t-t_j)(t-t_{j+1}), \xi_j \in (t_{j-1}, t_{j+1}).
\end{align*}
On  $I_0=[t_0,t_1]$, $y(t)$ is interpolated by using the grid points $t_0,t_{1/2}$ and $t_1$
\begin{equation*}\label{quad_y0}
y(t) = L^2_0 y(t)+R^2_0(t),
\end{equation*}
where
\begin{align*}
	 L^2_0 y(t)&=Q^0_0 y(t_{0}) +Q^0_{\frac{1}{2}}(t)y(t_{\frac{1}{2}})+ Q^0_1(t) y(t_1), \quad
	 R^2_0 y(t)=\frac{y'''(\xi_0)}{6}(t-t_{1})(t-t_{1/2})(t-t_{1}), \xi_0 \in (t_{0}, t_{1}),\\ 
	 Q^0_0(t)&=\frac{(t-t_{\frac{1}{2}})(t-t_1)}{(t_0-t_{\frac{1}{2}})(t_0-t_1)},
	 Q^0_{1/2}(t)=\frac{(t-t_{0})(t-t_1)}{(t_{\frac{1}{2}}-t_{0})(t_{\frac{1}{2}}-t_1)},
	 Q^0_1(t)=\frac{(t-t_0)(t-t_{\frac{1}{2}})}{(t_1-t_{0})(t_1-t_{\frac{1}{2}})}.
 \end{align*}
Thus, from (\ref{lin_sol2}), we have
\begin{align} \label{quad_sol}
	y(t_{n+1})&=g(t_{n+1}, y(t_{n+1}))+\beta [ A^{0,0}_{n+1}y(t_{0})+A^{1,0}_{n+1} y(t_{1/2})+A^{2,0}_{n+1} y(t_{1})]\\
		&+\beta \sum_{j=1}^{n}[A^{0,j}_{n+1}y(t_{j-1})+A^{1,j}_{n+1} y(t_{j})+A^{2,j}_{n+1}y(t_{j+1})]+T_{n+1}. \nonumber
\end{align}
where
\begin{align*}
A^{0,0}_{n+1}&=\int_{t_0}^{t_{1}}Q^0_0(s) \exp\left[ -\beta(t_{n+1}-s)\right]ds,\\
A^{1,0}_{n+1}&=\int_{t_0}^{t_{1}}Q^0_{1/2}(s)\exp\left[ -\beta(t_{n+1}-s)\right]ds,\\
A^{2,0}_{n+1}&=\int_{t_0}^{t_{1}}Q^0_1(s)\exp\left[ -\beta(t_{n+1}-s)\right]ds,
\end{align*}
and for $1\le j\le n$,
\begin{align*}
A^{i,j}_{n+1}&=\int_{t_j}^{t_{j+1}}Q^j_{i+j-1}(s)\exp\left[ -\beta(t_{n+1}-s)\right]ds, \quad i=0,1,2, \\
 T_{n+1}&=\beta \sum_{j=0}^{n} \int_{t_{j}}^{t_{j+1}}  R^2_{j}(y(s))\exp\left[ -\beta(t_{n+1}-s) \right]ds.
\end{align*}
Now, the value of $f(t_{n+1}, y(t_{n+1}))$ can be approximated by using the quadratic interpolation of $f(t,y(t))$ with grid points $t_{n-2}, t_{n-1}$ and $t_{n}$:
\begin{equation} \label{quad_f}
	f(t_{n+1}, y(t_n)) = f(t_{n-2},y(t_{n-2}))-3 f(t_{n-1}, y(t_{n-1}))+3f(t_{n},y(t_{n}))+f'''(\eta,y(\eta)) h^3, \quad \eta \in (t_{n-2},t_{n}).
\end{equation}
Thus we propose the predictor-corrector scheme as follow
\begin{equation}\label{pecequad}
\begin{split} 
	\left[ 1-\beta A^{2,n}_{n+1}\right] y_{n+1}
	&=g(t_{n+1}, y^p_{n+1})+\beta( A^{0,0}_{n+1}y_{0}+A^{1,0}_{n+1} y_{1/2}+A^{2,0}_{n+1} y_{1}) \\
		&+\beta \left[\left\{ \sum_{j=1}^{n-1}(A^{0,j}_{n+1}y_{j-1}+A^{1,j}_{n+1} y_{j}+A^{2,j}_{n+1}y_{j+1})\right\}+
		A^{0,n}_{n+1}y_{n-1}+A^{1,n}_{n+1} y_{n}\right] \\
	\left[ 1-\beta A^{2,n}_{n+1}\right] y^P_{n+1}
	&=g_{n+1}+\beta( A^{0,0}_{n+1}y_{0}+A^{1,0}_{n+1} y_{1/2}+A^{2,0}_{n+1} y_{1})   \\
		&+\beta \left[\left\{ \sum_{j=1}^{n-1}(A^{0,j}_{n+1}y_{j-1}+A^{1,j}_{n+1} y_{j}+A^{2,j}_{n+1}y_{j+1})\right\}+
		A^{0,n}_{n+1}y_{n-1}+A^{1,n}_{n+1} y_{n}\right] ,  
\end{split}
\end{equation}
where 
\begin{equation}\label{q_gnp}
	g_{n+1}=\frac{1-\alp}{M(\alp)}  (f_{n-2}-3 f_{n-1}+3f_{n})+y_0 \exp\left[ -\beta t_{n+1}\right], n\ge 2.
\end{equation}
\begin{rem}
	We calculate $y_1$ and $y_2$ using the start-up algorithm described in \ref{appendix-start-up}.
\end{rem}
\subsection{Error Analysis of Quadratic Interpolation}
By the similar procedure in the linear case, we can obtain the error analysis of the scheme using the quadratic interpolation. 
A simple calculation of $A^{2,n}_{n+1}$ and the Taylor expansion of the exponential function provides
\begin{equation*}
1-\beta A^{2,n}_{n+1}=1-\frac{5}{12} \beta h+\frac{1}{8}\beta^2 h^2+\mathcal{O}(h^3).
\end{equation*}
Then it is easy to see that for all $\alp \in(0,1)$
\begin{equation}\label{q_h}
	1-\frac{5}{12}\beta  h+\frac{1}{8}\beta^2 h^2 \ge \frac{47}{72}.
\end{equation}  
\begin{lem}\label{lem_q1}
For $j=0,\cdots,n$ and $i=0,1,2$,
\begin{equation*}
| A^{i,j}_{n+1}| \le h, \quad  |T_{n+1}| \le Ch^3, 
\end{equation*}
where $C$ is an independent constant.
\end{lem}

\begin{thm}\label{thm_pre_qglo}(Global Error)
Suppose that $f(\cdot, u(\cdot))\in\mathcal{C}^3[0,T]$, and furthermore is Lipschitz continuous in the second argument in (\ref{lipf}), then there exist a small $h \ll 1$ such that  $|E^P_{n+1}|$ is
\begin{equation*}
|E^P_{n+1} |\le 2 \beta h |E_{1/2}|+\frac{2L(1-\alp)}{|M(\alp)|}  (|E_{n-2}|+3|E_{n-1}|+3|E_n | ) )
	+6 \beta h \sum_{j=1}^{n}|E_{j}|+\mathcal{O}(h^3).
\end{equation*}
\end{thm}
\begin{proof}
Applying the Lipschitz condition of $f$ and using (\ref{quad_f}) and (\ref{q_gnp}), we have 
\begin{equation}\label{q_gerr}
	|g(t_{n+1}, y(t_{n+1}))-g_{n+1}| \le \frac{1-\alp}{|M(\alp)|} L  (|E_{n-2}|+3|E_{n-1}|+3|E_n | )+\mathcal{O}(h^3).
\end{equation}
From (\ref{q_h}), we can choose $h \ll 1$ for the given $\alp \in (0,1)$ such that
\begin{equation}\label{q_apph}
	\frac{1}{2} \le 1-\beta A^{2,n}_{n+1}.
\end{equation}
For the above $h$, subtracting  (\ref{pecequad}) from (\ref{quad_sol}) combined with (\ref{q_gerr}) and (\ref{q_h}), we have 
\begin{align*}
	|E^P_{n+1}|
	&\le \frac{2L(1-\alp)}{|M(\alp)|}  (|E_{n-2}|+3|E_{n-1}|+3|E_n | ) +\mathcal{O}(h^3)+
	2 \beta(A^{1,0}_{n+1} |E_{1/2}|+A^{2,0}_{n+1} |E_{1}|)  \ \\
		&+2 \beta \left[\left\{ \sum_{j=1}^{n-1}(A^{0,j}_{n+1}|E_{j-1}|+A^{1,j}_{n+1} |E_{j}|+A^{2,j}_{n+1}|E_{j+1}|)\right\}+
		A^{0,n}_{n+1}|E_{n-1}|+A^{1,n}_{n+1} |E_{n}|\right].
\end{align*}
By Lemma \ref{lem_q1}, we deduce that
\begin{equation*}
	|E^P_{n+1}|
	\le 2 \beta h |E_{1/2}|+\frac{2L(1-\alp)}{|M(\alp)|}  (|E_{n-2}|+3|E_{n-1}|+3|E_n | ) 
	+6 \beta h \sum_{j=1}^{n}|E_{j}|+\mathcal{O}(h^3).
\end{equation*}
\end{proof}

\begin{thm}\label{thm_q_glo}(Global Error)
With the same assumptions as those of Theorem $\ref{thm_pre_qglo}$, 
then there exist a small $h \ll 1$ such that  global error $|E_{n+1}|$ is
\begin{equation*}
	 |E_{n+1}|
	\le \left\{ C_1(|E_{n-2}|+ 3|E_{n-1}|+3|E_n|)+C_2 h |E_{1/2}|+C_3h^3\right\}
	     \exp \left[ C_4  T \right],
\end{equation*}
where $C_i, i=1,...,4$ are independent constant of all grid parameters.
\end{thm}
\begin{proof}
For the choice of $h \ll 1$ in (\ref{q_apph}), subtracting  (\ref{pecequad}) from (\ref{quad_sol}) and using the Lipschitz condition of $f$ in (\ref{lipf}), we have
\begin{equation} \label{q_enp1}
	 |E_{n+1}|
	\le \frac{2L(1-\alp)}{|M(\alp)|} |E^P_{n+1}|+2 \beta h |E_{1/2}|+6 \beta h \sum_{j=1}^{n}|E_{j}|+2|T_{n+1}|.
\end{equation}
By substituting $E^P_{n+1}$ in Theorem \ref{thm_pre_qglo} into (\ref{q_enp1}), we deduce 
\begin{align*}
	 |E_{n+1}|
	&\le \left( \frac{2L(1-\alp)}{M(\alp)} \right)^2(|E_{n-2}|+ 3|E_{n-1}|+3|E_n|)
	+2 \beta \left( \frac{2L(1-\alp)}{|M(\alp)|} +1\right) h |E_{1/2}|\\
		&+ 6\beta\left( \frac{2L(1-\alp)}{|M(\alp)|} +1\right) h\sum_{j=0}^n |E_j| +Ch^3,
\end{align*}
where $C$ is a generic constant.
By the Discrete Gronwall's inequality in Lemma \ref{lem_Gronwall} , we deduce 
\begin{equation*}
	 |E_{n+1}|
	\le \left\{ \left( \kappa(\alp) \right)^2(|E_{n-2}|+ 3|E_{n-1}|+3|E_n|)
	+2\beta \left( \kappa(\alp) +1\right)h |E_{1/2}|+Ch^3\right\}
	     \exp \left[ 4\beta\left(\kappa(\alp)+1\right)  T \right],
\end{equation*}
where $\displaystyle{\kappa(\alp)=\frac{2L(1-\alp)}{|M(\alp)|}}.$
\end{proof}
Applying the mathematical induction on $n$ in the Theorem \ref{thm_q_glo},
we have the following theorem
\begin{thm}\label{thm_q_err}(Global Error)
With the same assumptions as those of Theorem \ref{thm_q_glo}, we have
\begin{equation*}
	 |E_{n+1}| \le C h^3,
\end{equation*}
given $|E_{1/2}| \le C_1 h^2$, and $|E_1| \le C_2h^3$, $|E_2| \le C_3h^3$.
\end{thm}

If we apply the quadratic interpolation of $f(t,y(t))$ with grid points $t_0,t_{1/2}$ and $t_1$,
it is easy to see that the value of $f(t_2,y(t_2))$ can be evaluated by
\begin{equation*}\label{appf_t12}
	f(t_2,y(t_2))=3f(t_0,y(t_0))-8f(t_{1/2},y(t_{1/2}))+6f(t_1,y(t_1))+\frac{f'''(\eta,y(\eta))}{2} h^3,\quad \eta \in(t_0,t_1).
\end{equation*}
It implies
\begin{equation*}\label{appg2}
	|g(t_{2}, y(t_{2}))-g(t_{2}, y^p_{2})|
	\le \frac{L(1-\alp)}{|M(\alp)|}(3|E_0|+8|E_{1/2}|+6|E_1|)+\mathcal{O}(h^3).
\end{equation*}
Then we have the following estimates
\begin{equation}\label{err2}
	\begin{split}
	|E^p_2| &\le  2 \beta h |E_{1/2}|+\frac{2L(1-\alp)}{|M(\alp)|}  (8|E_{1/2}|+6|E_1 | ) )
	+4\beta h |E_{1}|+\mathcal{O}(h^3),\\
	 |E_{2}|
	&\le \left( \frac{2L(1-\alp)}{M(\alp)} \right)^2( 8|E_{1/2}|+6|E_1|)
	+2 \beta \left( \frac{2L(1-\alp)}{|M(\alp)|} +1\right) h |E_{1/2}|\\
		&+ 4\beta h\left( \frac{2L(1-\alp)}{|M(\alp)|} +1\right)  |E_1| +Ch^3.
 	\end{split}
\end{equation}
Using (\ref{err2}) combined with (\ref{q_enp1}), we have the following global error analysis
\begin{thm}\label{thm_q_err2}(Global Error)
If $|E_{1/2}| \le C_1 h^3$, and $|E_1| \le C_2h^3$, then 
\begin{equation}\label{glb-err-quad}
	 |E_{n+1}| \le C h^3.
\end{equation}
\end{thm}
\section{Fast Algorithm for Computation of Memory}\label{sec_fast}
In this section, we describe how the efficiency of the proposed predictor-corrector schemes can be improved. First, we rewrite the exact solution $y(t)$ at $t=t_{n+1}$ in  (\ref{model}) as
\begin{equation}\label{re_y}
	y(t_{n+1})=g(t_{n+1},y(t_{n+1}))+\beta \int_{t_0}^{t_{n}} y(s) \exp\left[ -\beta(t_{n+1}-s)\right]ds+\beta \int_{t_n}^{t_{n+1}} y(s) \exp\left[ -\beta(t_{n+1}-s)\right]ds.
\end{equation}
Here, the second term in the right hand side of $(\ref{re_y})$ is called a memory part and the last term a local part, respectively.
For the predictor-corrector scheme with a linear interpolation in (\ref{pecelin}), we have the approximations of the memory and local parts as follows;
\begin{equation*}
	\text{Memory} \approx \beta \left\{\sum_{j=0}^{n-1}(B^{1,j}_{n+1}y_{j}+B^{2,j}_{n+1} y_{j+1})\right\},\quad
	\text{Local} \approx \beta B^{1,n}_{n+1}y_{n}.
\end{equation*}
It is worth noting that the computational complexity required to approximate the memory term is extremely increasing as $n$ is increasing.
In what follow, we propose a recurrence relation that reduces the number of arithmetic operations required in the algorithms of the meory term.
To do that, let us define
\begin{equation*}
	Y_{mem}(t_{k})=\int_{t_0}^{t_{k-1}} y(s) \exp\left[ -\beta(t_{k}-s)\right]ds.
\end{equation*}
Then we have
\begin{align*}
	Y_{mem}(t_{n+1})
	&= \int_{t_0}^{t_{n}} y(s) \exp\left[ -\beta(t_{n+1}-s)\right]ds\\ 
	 =&\int_{t_0}^{t_{n-1}} y(s) \exp\left[ -\beta(t_{n}+h-s)\right]ds+\int_{t_{n-1}}^{t_{n}} y(s) \exp\left[ -\beta(t_{n+1}-s)\right]ds\\
	 =& \exp\left[-\beta h\right] \int_{t_0}^{t_{n-1}} y(s) \exp\left[ -\beta(t_{n}-s)\right]ds+\int_{t_{n-1}}^{t_{n}} y(s) \exp\left[ -\beta(t_{n+1}-s)\right]ds,
\end{align*}
which  gives the following recurrence relation for $Y_{mem}(t_{n+1})$
\begin{equation}\label{recurrence}
	Y_{mem}(t_{n+1})=\exp\left[-\beta h\right]Y_{mem}(t_{n})+\int_{t_{n-1}}^{t_{n}} y(s) \exp\left[ -\beta(t_{n+1}-s)\right]ds.
\end{equation}
Using the linear interpolation of $y(t)$ on $[t_{n-1},t_n]$, the approximated value $Y_{mem,n+1}$ can be evaluated by
\begin{equation*}
	Y_{mem,n+1}=\exp\left[-\beta h\right]Y_{mem,n}+B^{1,n-1}_{n+1} y_{n-1}+B^{2,n-1}_{n+1} y_{n}.
\end{equation*}
Then we have the following fast linear predictor-corrector scheme
\begin{equation}\label{fast_pecelin}
\begin{split}
	\left[ 1-\beta B^{2,n}_{n+1}\right] y_{n+1}&=g(t_{n+1}, y^p_{n+1})
		+\beta \left[ Y_{mem,n+1}+B^{1,n}_{n+1}y_{n}\right] \\
	\left[ 1-\beta B^{2,n}_{n+1}\right] y^p_{n+1}&=g_{n+1}
		+\beta \left[ Y_{mem,n+1}+B^{1,n}_{n+1}y_{n}\right].
\end{split}
\end{equation}

Employing the quadratic interpolation of $y(t)$ on $[t_{n-1},t_n]$, we can obtain the approximated value $Y_{mem,n+1}$ and the following fast
quadratic predictor-corrector scheme is proposed;
\begin{equation}\label{fast_pecequad}
\begin{split} 
	\left[ 1-\beta A^{2,n}_{n+1}\right] y_{n+1}
	&=g(t_{n+1}, y^p_{n+1})
		+\beta \left[ Y_{mem,n+1}+A^{0,n}_{n+1}y_{n-1}+A^{1,n}_{n+1} y_{n}\right]  \\  
	\left[ 1-\beta A^{2,n}_{n+1}\right] y^p_{n+1}
	&=g_{n+1}+\beta \left[ Y_{mem,n+1} +A^{0,n}_{n+1}y_{n-1}+A^{1,n}_{n+1} y_{n}\right] .
\end{split}
\end{equation}

\begin{rem}
The proposed fast schemes $(\ref{fast_pecelin})$ and $(\ref{fast_pecequad})$ do not contain any summation operations while the standard algorithms $(\ref{pecelin})$ and $(\ref{pecequad})$ have. Considering that the summation operation requires approximately $\mathcal{O}{(n^2)}$ arithmetic operations for the given step $n$, the proposed fast algorithms can reduce the required arithmetic operation to $\mathcal{O}{(n)}$. This will be confirmed by numerical examples in the following section.
\end{rem}

\section{Numerical Results}\label{sec_num}
In this section, four numerical examples are introduced. In the examples, we demonstrate the performance of the proposed methods. The global error estimates \eqref{glb-err-linear} and \eqref{glb-err-quad} are justified by observing the computed convergence rates from the examples. The effectiveness of the fast algorithm is also confirmed by the examples. \\
we denote numerical methods that we suggested as follows:
\begin{enumerate}[(1)]
	\item C-PC-L : Predictor-Corrector method with linear interpolation
	\item C-PC-Q : Predictor-Corrector method with quadratic interpolation
	\item F-PC-L : Fast Predictor-Corrector method with linear interpolation
	\item F-PC-Q : Fast Predictor-Corrector method with quadratic interpolation
\end{enumerate} 
\subsection{Examples of Nonlinear Fractional Differential Equations}
To measure the error of approximate solution, we use the following error estimates:
\begin{itemize}
	\item \emph{Maximum error over $[0,T]$:}
	\begin{align*}
	E_{\max}=\max_{0\leq n\leq N}|E_n|. \quad \text{i.e.} \ \max_{0\leq n\leq N}|y(t_n)-y_{n}|.
	\end{align*}
	\item  \emph{$L^2$ error over $[0,T]$:}
	\begin{align*}
	E_{L^2}=\left(h\sum_{n=0}^N|E_n|^2\right)^{1/2}, \quad \text{i.e.} \ \left(h\sum_{n=0}^N|y(t_n)-y_n|^2\right)^{1/2}.
	\end{align*}
\end{itemize}

\begin{example}\label{ex2}
	\begin{align*}
	D^{\alp}_{0}y(t)=&\begin{dcases}
	\frac{-M(\alp)}{\beta(\beta-1)(\alp-1)}\left((e^{-\beta t}-1)-\beta(e^{-t}-1)\right)+y^2-(e^{-t}-1+t)^2,&\text{if }\alp\in(0,1)\setminus\{0.5\}\\
	-2M(\alp)(e^{-t}-1+te^{-t}),&\text{if }\alp=0.5,
	\end{dcases}\\
	y(0)=&0.
	\end{align*}
	where $M(\alp)=1$, the exact solution is $y(t)=e^{-t}-1+t$, and $T=1$.
\end{example}

\begin{example}\label{ex3}
	\begin{align*}
	D^{\alp}_{0}y(t)=&\frac{M(\alp)}{(\beta^2+1)^2(\alp-1)}(\beta^3(e^{-\beta  t}-\cos{t}+t\sin{t})-\beta^2(2\sin{t}+t\cos{t})\\
	&\hspace{75pt}-t\cos{t}+\beta(\cos{t}-e^{-\beta t}+t\sin{t}))+y^2-t^2\cos^2{t}, \quad t\in[0,1]\\
	y(0)=&0.
	\end{align*}
	Here $M(\alp)=1$ and the exact solution is $y(t)=t\cos{t}$.
\end{example}

In Example \ref{ex2} and \ref{ex3}, since $f(t,y)\in\mathcal{C}^{\infty}([0,1])$, the convergence rate against the size of the grid $h$ for the linear interpolation and the quadratic interpolation are quadratic and cubic, respectively. From Table \ref{TABLE-ex2}, \ref{TABLE-ex3}, \ref{TABLE-CF_time} and Figure \ref{FIGURE-CF_time} (a) and (b), we observe the following:
\begin{enumerate}
	\item Maximum errors, $L^2$ errors, and rates of convergence against $h$ of approximate solutions obtained by predictor-corrector methods and fast predictor-corrector methods are shown in Table \ref{TABLE-ex2}. As it shown in Table \ref{TABLE-ex2}, errors obtained by C-PC-L and F-PC-L, C-PC-Q and F-PC-Q are exactly identical, respectively. Because the memory term in the predictor-corrector methods is only replaced by the recurrence relation \ref{recurrence} in the fast predictor-corrector algorithm. Table \ref{TABLE-ex3} illustrates maximum errors, $L^2$ errors, and rates of convergence versus $h$ of approximate solutions obtained by the F-PC-L and F-PC-Q. Both computed convergence profiles for $E_{\max}$ and $E_{L^2}$ are around 2 in case of linear interpolation and 3 in case of quadratic interpolation, respectively in Table \ref{TABLE-ex2} and \ref{TABLE-ex3}.
	\item In Table \ref{TABLE-ex2}, \ref{TABLE-ex3}, we can see that the computed convergence profiles of the maximum errors support the theoretical convergence rates for $\alpha=0.2,0.5,0.8$.
	\item Table \ref{TABLE-CF_time} shows CPU times in second executed the proposed methods with quadratic interpolation. We can see that the computational cost is remarkably reduced as the fast algorithm is implemented, comparing with the computational cost of the predictor-corrector scheme.
	\item The plots exhibited in Figure \ref{FIGURE-CF_time} display CPU time versus the number of grid points $N=1/h$ in Table \ref{TABLE-CF_time} with rates of CPU time increasing. The rate of CPU time executed C-PC-Q is around 2 for both Example \ref{ex2} and \ref{ex3} while the rate of CPU time executed F-PC-Q is around 1 for both examples. It means that the CPU time executed C-PC-Q is increased exponentionally while the CPU time executed F-PC-Q is increased linearlly. Thus the effectiveness of the fast algorithm is verified. The difference between CPU time executed C-PC-Q and F-PC-Q is distinctly appeared in examples of time-fractional partial differential equations \ref{ex4} and \ref{ex5} and will be addressed in the subsequent subsection.
\end{enumerate}

\begin{table}[H]
	\centering
	\begin{tabular}{l|llllV{2}llll}
		\toprule
		\toprule
		& \multicolumn{4}{cV{2}}{C-PC-L}   &   \multicolumn{4}{c}{F-PC-L}       \\
		\cline{2-9}
		\cline{2-9}
		& $E_{\max}$          & roc   & $E_{L^2}$ & roc         & $E_{\max}$ & roc   & $E_{L^2}$ & roc   \\
		\hline
		$h$& \multicolumn{8}{c}{$\alpha=0.2$}       \\
		\hline
		1/10  & 1.96E-03 & - & 7.95E-04 & - & 1.96E-03 & - & 7.95E-04 & - \\
		1/20  & 4.85E-04 & 2.01 & 1.76E-04 & 2.18 & 4.85E-04 & 2.01 & 1.76E-04 & 2.18 \\
		1/40  & 1.20E-04 & 2.02 & 4.09E-05 & 2.11 & 1.20E-04 & 2.02 & 4.09E-05 & 2.11 \\
		1/80  & 2.97E-05 & 2.01 & 9.80E-06 & 2.06 & 2.97E-05 & 2.01 & 9.80E-06 & 2.06 \\
		1/160 & 7.37E-06 & 2.01 & 2.40E-06 & 2.03 & 7.37E-06 & 2.01 & 2.40E-06 & 2.03 \\
		1/320 & 1.84E-06 & 2.01 & 5.92E-07 & 2.02 & 1.84E-06 & 2.01 & 5.92E-07 & 2.02 \\
		\hline
		$h$ & \multicolumn{8}{c}{$\alpha=0.5$}       \\
		\hline
		1/10  & 5.19E-04 & - & 3.54E-04 & - & 5.19E-04 & - & 3.54E-04 & - \\
		1/20  & 1.31E-04 & 1.98 & 8.74E-05 & 2.02 & 1.31E-04 & 1.98 & 8.74E-05 & 2.02 \\
		1/40  & 3.29E-05 & 2.00 & 2.16E-05 & 2.01 & 3.29E-05 & 2.00 & 2.16E-05 & 2.01 \\
		1/80  & 8.23E-06 & 2.00 & 5.38E-06 & 2.01 & 8.23E-06 & 2.00 & 5.38E-06 & 2.01 \\
		1/160 & 2.06E-06 & 2.00 & 1.34E-06 & 2.00 & 2.06E-06 & 2.00 & 1.34E-06 & 2.00 \\
		1/320 & 5.14E-07 & 2.00 & 3.34E-07 & 2.00 & 5.14E-07 & 2.00 & 3.34E-07 & 2.00\\
		\hline
		$h$ & \multicolumn{8}{c}{$\alpha=0.8$}       \\
		\hline
		1/10  & 2.58E-03 & - & 1.57E-03 & - & 2.58E-03 & - & 1.57E-03 & - \\
		1/20  & 6.82E-04 & 1.92 & 4.01E-04 & 1.97 & 6.82E-04 & 1.92 & 4.01E-04 & 1.97 \\
		1/40  & 1.73E-04 & 1.98 & 1.00E-04 & 2.00 & 1.73E-04 & 1.98 & 1.00E-04 & 2.00 \\
		1/80  & 4.36E-05 & 1.99 & 2.50E-05 & 2.01 & 4.36E-05 & 1.99 & 2.50E-05 & 2.01 \\
		1/160 & 1.09E-05 & 2.00 & 6.22E-06 & 2.00 & 1.09E-05 & 2.00 & 6.22E-06 & 2.00 \\
		1/320 & 2.73E-06 & 2.00 & 1.55E-06 & 2.00 & 2.73E-06 & 2.00 & 1.55E-06 & 2.00 \\
		\toprule
		\toprule
		& \multicolumn{4}{cV{2}}{C-PC-Q}   &   \multicolumn{4}{c}{F-PC-Q}       \\
		\cline{2-9}
		& $E_{\max}$          & roc   & $E_{L^2}$ & roc         & $E_{\max}$ & roc   & $E_{L^2}$ & roc   \\
		\hline
		$h$& \multicolumn{8}{c}{$\alpha=0.2$}       \\
		\hline
		1/10  & 5.34E-04 & - & 2.68E-04 & - & 5.34E-04 & - & 2.68E-04 & - \\
		1/20  & 6.23E-05 & 3.10 & 2.92E-05 & 3.20 & 6.23E-05 & 3.10 & 2.92E-05 & 3.20 \\
		1/40  & 7.48E-06 & 3.06 & 3.39E-06 & 3.11 & 7.48E-06 & 3.06 & 3.39E-06 & 3.11 \\
		1/80  & 9.15E-07 & 3.03 & 4.08E-07 & 3.06 & 9.15E-07 & 3.03 & 4.08E-07 & 3.06 \\
		1/160 & 1.13E-07 & 3.01 & 5.00E-08 & 3.03 & 1.13E-07 & 3.01 & 5.00E-08 & 3.03 \\
		1/320 & 1.40E-08 & 3.02 & 3.57E-09 & 3.81 & 1.40E-08 & 3.02 & 3.57E-09 & 3.81 \\
		\hline
		$h$ & \multicolumn{8}{c}{$\alpha=0.5$}       \\
		\hline
		1/10  & 3.78E-04 & - & 2.10E-04 & - & 3.78E-04 & - & 2.10E-04 & - \\
		1/20  & 3.99E-05 & 3.24 & 2.13E-05 & 3.30 & 3.99E-05 & 3.24 & 2.13E-05 & 3.30 \\
		1/40  & 4.55E-06 & 3.13 & 2.38E-06 & 3.16 & 4.55E-06 & 3.13 & 2.38E-06 & 3.16 \\
		1/80  & 5.42E-07 & 3.07 & 2.80E-07 & 3.08 & 5.42E-07 & 3.07 & 2.80E-07 & 3.08 \\
		1/160 & 6.61E-08 & 3.04 & 3.40E-08 & 3.04 & 6.61E-08 & 3.04 & 3.40E-08 & 3.04 \\
		1/320 & 7.82E-09 & 3.08 & 4.08E-09 & 3.06 & 7.82E-09 & 3.08 & 4.08E-09 & 3.06\\
		\hline
		$h$ & \multicolumn{8}{c}{$\alpha=0.8$}       \\
		\hline
		1/10  & 7.08E-05 & - & 5.52E-05 & - & 7.07E-05 & - & 5.51E-05 & - \\
		1/20  & 2.88E-06 & 4.62 & 1.91E-06 & 4.85 & 2.88E-06 & 4.62 & 1.91E-06 & 4.85 \\
		1/40  & 8.25E-07 & 1.80 & 2.92E-07 & 2.71 & 8.25E-07 & 1.80 & 2.92E-07 & 2.71 \\
		1/80  & 1.31E-07 & 2.66 & 4.98E-08 & 2.55 & 1.31E-07 & 2.66 & 4.98E-08 & 2.55 \\
		1/160 & 1.80E-08 & 2.86 & 7.08E-09 & 2.81 & 1.80E-08 & 2.86 & 7.08E-09 & 2.81 \\
		1/320 & 2.28E-09 & 2.98 & 9.15E-10 & 2.95 & 2.28E-09 & 2.98 & 9.15E-10 & 2.95 \\
		\bottomrule
		\bottomrule
	\end{tabular}
	\caption{Maximum errors and $L^2$ errors with computed rates of convergence obtained by the proposed predictor-corrector methods and fast algorithms in Example \ref{ex2} with $\alpha=0.2, 0.5, 0.8$}
	\label{TABLE-ex2}
\end{table}

\begin{table}[H]
	\centering
	\begin{tabular}{l|llllV{2}llll}
		\toprule
		\toprule
		& \multicolumn{4}{cV{2}}{F-PC-L}   &   \multicolumn{4}{c}{F-PC-Q}       \\
		\cline{2-9}
		\cline{2-9}
		& $E_{\max}$          & roc   & $E_{L^2}$ & roc         & $E_{\max}$ & roc   & $E_{L^2}$ & roc   \\
		\hline
		$h$& \multicolumn{8}{c}{$\alpha=0.2$} \\
		\hline
		1/10  & 2.80E-01 & -    & 1.03E-01 & -    & 2.32E-02 & -    & 1.06E-02 & - \\
		1/20  & 7.12E-02 & 1.97 & 2.55E-02 & 2.02 & 2.07E-03 & 3.49 & 1.04E-03 & 3.34 \\
		1/40  & 1.23E-02 & 2.53 & 5.34E-03 & 2.25 & 1.96E-04 & 3.40 & 1.04E-04 & 3.33 \\
		1/80  & 2.59E-03 & 2.25 & 1.19E-03 & 2.16 & 2.24E-05 & 3.13 & 1.22E-05 & 3.09 \\
		1/160 & 6.21E-04 & 2.06 & 2.89E-04 & 2.04 & 2.72E-06 & 3.04 & 1.48E-06 & 3.04 \\
		1/320 & 1.54E-04 & 2.01 & 7.17E-05 & 2.01 & 1.61E-07 & 4.08 & 6.33E-08 & 4.54 \\
		\hline
		$h$ & \multicolumn{8}{c}{$\alpha=0.5$}       \\
		\hline
		1/10  & 1.90E-02 & -    & 9.37E-03 & -    & 1.00E-03 & -    & 7.43E-04 & - \\
		1/20  & 4.55E-03 & 2.06 & 2.23E-03 & 2.07 & 9.72E-05 & 3.37 & 6.90E-05 & 3.43 \\
		1/40  & 1.13E-03 & 2.01 & 5.47E-04 & 2.03 & 1.04E-05 & 3.22 & 7.22E-06 & 3.26 \\
		1/80  & 2.82E-04 & 2.00 & 1.36E-04 & 2.01 & 1.20E-06 & 3.12 & 8.24E-07 & 3.13 \\
		1/160 & 7.08E-05 & 2.00 & 3.39E-05 & 2.00 & 1.44E-07 & 3.06 & 9.85E-08 & 3.06 \\
		1/320 & 1.77E-05 & 2.00 & 8.46E-06 & 2.00 & 1.73E-08 & 3.06 & 1.19E-08 & 3.05\\
		\hline
		$h$ & \multicolumn{8}{c}{$\alpha=0.8$}       \\
		\hline
		1/10  & 4.03E-03 & -    & 2.13E-03 & -    & 5.14E-03 & -    & 3.06E-03 & - \\
		1/20  & 9.80E-04 & 2.04 & 4.74E-04 & 2.17 & 5.95E-04 & 3.11 & 3.42E-04 & 3.16 \\
		1/40  & 2.36E-04 & 2.05 & 1.08E-04 & 2.14 & 6.94E-05 & 3.10 & 3.90E-05 & 3.13 \\
		1/80  & 5.74E-05 & 2.04 & 2.53E-05 & 2.09 & 8.31E-06 & 3.06 & 4.61E-06 & 3.08 \\
		1/160 & 1.41E-05 & 2.02 & 6.11E-06 & 2.05 & 1.01E-06 & 3.03 & 5.59E-07 & 3.04 \\
		1/320 & 3.50E-06 & 2.01 & 1.50E-06 & 2.03 & 1.25E-07 & 3.02 & 6.86E-08 & 3.03 \\
		\bottomrule
		\bottomrule
	\end{tabular}
	\caption{Maximum errors and $L^2$ errors with computed rates of convergence obtained by the proposed fast predictor-corrector methods with $\alpha = 0.2,0.5,0.8$ in Example \ref{ex3}}
\label{TABLE-ex3}
\end{table}

\begin{table}[H]
	\centering
	\begin{tabular}{c|cc|cc}
		\toprule
		\toprule
		& \multicolumn{2}{c|}{Example \ref{ex2}} & \multicolumn{2}{c}{Example \ref{ex3}} \\
		\midrule
		$h=(10\cdot 2^k)^{-1}$ & \multicolumn{4}{c}{CPU time} \\
		\midrule
		$k$ & C-PC-Q & F-PC-Q & C-PC-Q & F-PC-Q \\
		\midrule
		0	&	6.88E-03	&	6.22E-03	&	7.01E-03	&	4.78E-03
\\
		1	&	9.09E-03	&	8.26E-03	&	9.08E-03	&	6.14E-03
\\
		2	&	1.87E-03	&	4.29E-04	&	2.06E-03	&	1.58E-04
\\
		3	&	1.99E-03	&	2.70E-04	&	4.15E-03	&	2.60E-04
\\
		4	&	8.02E-03	&	4.43E-04	&	1.59E-02	&	5.69E-04
\\
		5	&	2.67E-02	&	6.90E-04	&	3.23E-02	&	1.01E-03 \\
		6	&	1.11E-01	&	1.07E-03	&	1.19E-01	&	1.28E-03	\\
		7	&	4.13E-01	&	2.10E-03	&	4.33E-01	&	2.05E-03	\\
		8	&	2.27E+00	&	3.87E-03	&	1.81E+00	&	3.92E-03	\\
		9	&	7.92E+00	&	1.15E-02	&	7.07E+00	&	7.73E-03	\\
		10	&	2.74E+01	&	1.79E-02	&	3.35E+01	&	1.65E-02	\\
		11	&	1.25E+02	&	3.11E-02	&	1.25E+02	&	2.99E-02	\\
		12	&	4.65E+02	&	5.72E-02	&	4.68E+02	&	5.82E-02	\\
		13	&	1.90E+03	&	1.11E-01	&	1.95E+03	&	1.18E-01	\\
		14	&	7.50E+03	&	2.60E-01	&	7.60E+03	&	2.23E-01	\\
		15	&	2.63E+04	&	4.66E-01	&	2.76E+04	&	5.01E-01	\\
		\bottomrule
		\bottomrule
	\end{tabular}
\caption{CPU times (second) exectued the predictor-corrector scheme  and fast predictor-corrector scheme with quadratic interpolation of examples \ref{ex2} and \ref{ex3} with $\alpha = 0.5$}
\label{TABLE-CF_time}
\end{table}

\begin{figure}[H]
	\begin{center}
		\subfloat[CPU time (CT) vs. $N$ for example \ref{ex2}]
		{\includegraphics[width=3in]{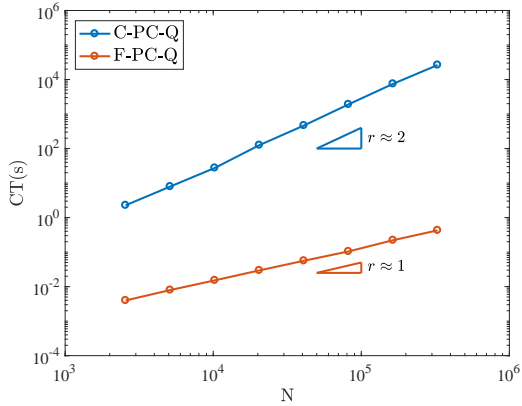}}
		\quad
		\subfloat[CPU time (CT) vs. $N$ for example \ref{ex3}]
		{\includegraphics[width=3in]{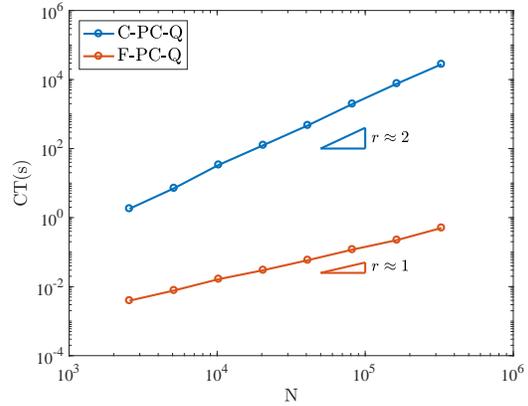}} \\
	\end{center}
	\caption{log-log plot of CPU time(CT) versus $N$ for examples \ref{ex2}, and \ref{ex3} with $\alpha=0.5$}
	\label{FIGURE-CF_time}
\end{figure}


\subsection{Examples of Nonlinear Time-fractional Partial Differential Equations}\label{sec_PDE}
In this subsection, we consider two time-fractional PDEs: time-fraction sub-diffusion and time-fractional advection diffusion equation. For each of diffusion equations, We derive the fast algorithm with quadratic interpolation combined with the central difference scheme for the spatial discretization. 
First, we consider the time-fractional advection diffusion equation with initial-boundary conditions:
\begin{align*}
&{}_0^{CF}D_t^\alpha y(x,t)+\frac{\partial}{\partial x}y(x,t)+\frac{\partial^2}{\partial x^2}y(x,t)=f(x,t,y),
\\ &y(x,0)=y_0(x),~y(a,t)=\psi(t),~y(b,t)=\phi(t),
\end{align*}
where $0<\alpha<1$ and $(x,t)\in[a,b]\times\mathbb{R}^+$.
To solve this problem, we consider the uniform grid
$\Phi_N^M=\{(x_m,t_n)|m=0,1,\cdots,M,~n=0,1,\cdots,N)\}$ with  $\Delta x_m=\tau$ and $\Delta t_n=h$
and $y_m^n$ is denoted by the approximated solution at $(x_m,t_n)$. For the space derivative, we approximate $y_x$ and $y_{xx}$ by the central difference as follows:
\begin{align*}
\frac{y(x+\tau)-y(x-\tau)}{2\tau}=y'(x)+\frac{\tau^2}{6}y^{(3)}(\xi),\quad \frac{y(x+\tau)-2y(x)+y(x-\tau)}{\tau^2}=y''(x)+\frac{\tau^2}{12}y^{(4)}(\xi),
\end{align*}
where $h>0$ and $\xi \in (x-h,x+h)$.
For the time derivative, we employ \eqref{fast_pecequad}. Then, the fast scheme for the advection diffusion equation is given by
\begin{align} 
\left[ 1-\beta A^{2,n}_{n+1}\right] y_{m}^{n+1}&+\frac{1-\alpha}{M(\alpha)}\frac{y_{m+1}^{n+1}-y_{m-1}^{n+1}}{2\tau}+\frac{1-\alpha}{M(\alpha)}\frac{y_{m+1}^{n+1}-2y_{m}^{n+1}+y_{m-1}^{n+1}}{\tau^2}\nonumber\\
&=g_{m}^{n+1,P}+\beta \left[ Y_{mem,m}^{n+1}+A^{0,n}_{n+1}y_{m}^{n-1}+A^{1,n}_{n+1} y_{m}^{n}\right], \label{eq:AD_cor}	
\\  
\left[ 1-\beta A^{2,n}_{n+1}\right] y^{n+1,P}_{m}&+\frac{1-\alpha}{M(\alpha)}\frac{y_{m+1}^{n+1,P}-y_{m-1}^{n+1,P}}{2\tau}+\frac{1-\alpha}{M(\alpha)}\frac{y_{m+1}^{n+1,P}-2y_{m}^{n+1,P}+y_{m-1}^{n+1,P}}{\tau^2}\nonumber\\
&=g_{m}^{n+1}+\beta \left[ Y_{mem,m}^{n+1} +A^{0,n}_{n+1}y_{m}^{n-1}+A^{1,n}_{n+1} y_{m}^{n}\right],\quad \text{for }1\leq m\leq M-1.\label{eq:AD_pre}
\end{align}
Similarly, we can deal with the time-fractional diffusion equation with initial-boundary conditions:
\begin{align*}
&{}_0^{CF}D_t^\alpha y(x,t)+\frac{\partial^2}{\partial x^2}y(x,t)=f(x,t,y),
\\&y(x,0)=y_0(x),~y(a,t)=\psi(t),~y(b,t)=\phi(t),
\end{align*}
where $0<\alpha<1$ and $(x,t)\in[a,b]\times\mathbb{R}^+$. By utilizing \eqref{eq:AD_cor} and \eqref{eq:AD_pre}, we obtain the following fast scheme for the diffusion equation:
\begin{align} 
\left[ 1-\beta A^{2,n}_{n+1}\right] y_{m}^{n+1}&+\frac{1-\alpha}{M(\alpha)}\frac{y_{m+1}^{n+1}-2y_{m}^{n+1}+y_{m-1}^{n+1}}{\tau^2}\nonumber \\
&=g_{m}^{n+1,P}+
\beta \left[ Y_{mem,m}^{n+1}+A^{0,n}_{n+1}y_{m}^{n-1}+A^{1,n}_{n+1} y_{m}^{n}\right], \label{eq:D_cor}
\\  
\left[ 1-\beta A^{2,n}_{n+1}\right] y^{n+1,P}_{m}&+\frac{1-\alpha}{M(\alpha)}\frac{y_{m+1}^{n+1,P}-2y_{m}^{n+1,P}+y_{m-1}^{n+1,P}}{\tau^2}\nonumber \\
&=g_{m}^{n+1}+\beta \left[ Y_{mem,m}^{n+1} +A^{0,n}_{n+1}y_{m}^{n-1}+A^{1,n}_{n+1} y_{m}^{n}\right].\label{eq:D_pre}
\end{align}
Now, we employ the fast predictor-corrector schemes \eqref{eq:AD_cor}-\eqref{eq:AD_pre} and \eqref{eq:D_cor}-\eqref{eq:D_pre} into Example \ref{ex4} and \ref{ex5}, respectively.

\begin{example}\label{ex4}
	\begin{align*}
	&{}_0^{CF}D_t^{\alp} y(x,t)+\frac{\partial}{\partial x}y(x,t)+\frac{\partial^2}{\partial x^2}y(x,t)=f(x,t,y),\\
	&y(x,0)=0,\quad y(0,t)=e^{-t}-1+t+t\sin{\pi t},\quad y(1,t)=-(e^{-t}-1+t+t\sin{\pi t}).
	\end{align*}
	where
	\begin{align*}
	f(x,t,y)&={}_0^{CF}D_t^{\alp}\left[\cos{3\pi x}(e^{-t}-1+t+t\sin{\pi t})\right]-(3\pi\sin{3\pi x}+9\pi^2\cos{3\pi x})(e^{-t}-1+t+t\sin\pi{t})\\
	&+y^2-(\cos{3\pi x}(e^{-t}-1+t+t\sin{\pi t}))^2.
	\end{align*}
	Its exact solution is $y(x,t)=\cos{3\pi x}(e^{-t}-1+t+t\sin{\pi t})$.
\end{example}

\begin{example}\label{ex5}
	\begin{align*}
&{}_0^{CF}D_t^{\alp} y(x,t)+\frac{\partial^2}{\partial x^2}y(x,t)=f(x,t,y),\\
&y(x,0)=x^2(x-1)^2,\quad y(0,t)=0,\quad y(1,t)=0,
\end{align*}
where
\begin{align*}
f(x,t,y)&={}_0^{CF}D_t^{\alp}\left[(1-t^4)x^2(x-1)^2\right]+(1-t^4)(2x^2+2(x-1)^2+4x(2x-2))\\
&+y^2-((1-t^4)x^2(x-1)^2)^2.
\end{align*}
Its exact solution is $y(x,t)=(1-t^4)x^2(x-1)^2$.
\end{example}
For both examples \ref{ex4} and \ref{ex5}, spatial domain is set $[0,1]$ and $T=1$. Errors of approximate solutions are estimated as follows:
\begin{itemize}
	\item \emph{Maximum error over the spatial domain $[a, b]$ at time $t_n$:}
	\begin{align*}
	E^x_{\max}=\max_{0\leq m\leq M}|y(x_m, t_n) - y^{n}_{m}|.
	\end{align*}
	\item \emph{Maximum error over the time interval $[0, T]$ at spatial grid $x_m$:}
	\begin{align*}
	E^t_{\max}=\max_{0\leq n\leq N}|y(x_m, t_n) - y^{n}_{m}|.
	\end{align*}
\end{itemize}
In Table \ref{TABLE-ex4-1}, \ref{TABLE-ex4-2}, \ref{TABLE-ex5-1}, and \ref{TABLE-ex5-2}, we observe the following:
\begin{enumerate}
	\item Table \ref{TABLE-ex4-1} and \ref{TABLE-ex5-1} shows maximum errors $E^x_{\max}$ and $E^t_{\max}$ with convergence rates of approximate solutions obtained by the fast predictor-corrector scheme for the advection diffusion equation \eqref{eq:AD_cor}-\eqref{eq:AD_pre} in Example \ref{ex4} and \eqref{eq:D_cor}-\eqref{eq:D_pre} in Example \ref{ex5} for $\alpha=0.2,0.5,0.8$, respectively.
	\item In Table \ref{TABLE-ex4-1} and \ref{TABLE-ex5-1}, we manipulate the spatial step size $\tau$ depending on the time step size $h$ and vice versa as follows
	\begin{eqnarray*}
		h &=& \tau^{2/3} \text{ as estimating } E^x_{\max},\\
		\tau &=& h^{3/2} \text{ as estimating } E^t_{\max}.
	\end{eqnarray*}
	In order to observe the global order of convergence obtained by our proposed methods then the global order of convergence is 
	\begin{eqnarray}\label{ex4-5-glb-order}
	\begin{aligned}
	\mathcal{O}(\tau^2 + h^3) &\approx \mathcal{O}(\tau^2) \text{ for } E^x_{\max},\\
	\mathcal{O}(\tau^2 + h^3) &\approx \mathcal{O}^(h^3) \text{ for } E^t_{\max}.
	\end{aligned}
	\end{eqnarray}
	 We can see that the computed convergence profiles agree the global orders of convergence \eqref{ex4-5-glb-order} for F-PC-Q so C-PC-Q does.
	\item Maximum errors $E^t_{\max}$, the computed convergence profiles, and CPU time executed C-PC-Q and F-PC-Q against $h$ and $\tau$ are shown in Table \ref{TABLE-ex4-2} and \ref{TABLE-ex5-2} for $\alpha=0.2,0.5,0.8$. To verify the performance of the proposed methods versus $h$ with fixed $\tau$ and vise versa, we fix $\tau$ and $h$ such that 
	\begin{eqnarray}
	\max h^3 > \tau^2 \text{ so that } \mathcal{O}(\tau^2 + h^3) &\approx& \mathcal{O}(h^3) \text{ for } E^t_{\max}, \label{ex4-5-glb-order-fix-tau}\\
	\max \tau^2 > h^3 \text{ so that } \mathcal{O}(\tau^2 + h^3) &\approx& \mathcal{O}(\tau^2) \text{ for } E^x_{\max}, \label{ex4-5-glb-order-fix-h}
	\end{eqnarray}
	respectively.
	\item In Example \ref{ex4}, we set $\tau=1/80000$ as estimating $E^t_{\max}$ with computed rates of convergence and CPU time executed. Since $(1/640)^3 > (1/80000)^2$, we expect the computed global slopes of convergence are approximately 3 by \eqref{ex4-5-glb-order-fix-tau} and Table \ref{TABLE-ex4-2} confirms the rates. On the other hand, we set $h=1/10000$ as estimating $E^x_{\max}$ to observe the performance of the fast predictor-corrector scheme \eqref{eq:AD_cor}-\eqref{eq:AD_pre} against $\tau$ only. Since $(1/640)^2 > (1/10000)^3$, the global order of convergence is 2 by \eqref{ex4-5-glb-order-fix-h} and we can see that the computed rates of convergence versus $\tau$ are nearly 2.
	\item In similar to that of Example \ref{ex4}, we set $\tau=1/40000$ and $h=1/4000$ as estimating $E^t_{\max}$ and $E^x_{\max}$ in Example \ref{ex5}, respectively. Since $(1/640)^3 > (1/40000)^2$ and $(1/640)^2 > (1/4000)^3$, the global orders of convergence are 3 and 2 by \eqref{ex4-5-glb-order-fix-tau} and \eqref{ex4-5-glb-order-fix-h}. Table \ref{TABLE-ex5-2} shows that the computed convergence profiles versus $h$ are nearly 3 except for the case of that $h=1/320, 1/640$, and versus $\tau$ are nearly 2.
	\item We can observe that C-PC-Q and F-PC-Q demonstrate a good performance and accuracy for solving time-fractional diffusion equations through Table \ref{TABLE-ex4-2} and \ref{TABLE-ex5-2}. Moreover, CPU times executed F-PC-Q are remarkably reduced comparing with that of C-PC-Q in Table \ref{TABLE-ex4-2} and \ref{TABLE-ex5-2}. It means that the reduction of the number of arithmetic operations using the recurrence relation \eqref{recurrence} is effective to reduce the computational cost for solving the PDE problems.
\end{enumerate}

\begin{table}[htp]
	\centering
	\resizebox{1\textwidth}{!}{
\begin{tabular}{l|llll|llll|llll}
	\toprule
	\toprule
		     & $E_{\max}^x$ & roc  & $E_{\max}^t$ & roc  & $E_{\max}^x$ & roc  & $E_{\max}^t$ & roc  & $E_{\max}^x$ & roc  & $E_{\max}^t$ & roc  \\
		     \cline{2-13}
	         &  \multicolumn{4}{c|}{$\alpha=0.2$}       & \multicolumn{4}{c|}{$\alpha=0.5$}       & \multicolumn{4}{c}{$\alpha=0.8$}       \\
	         \midrule
	$\tau,h$     & \multicolumn{12}{c}{F-PC-Q}  \\
	\midrule
	1/10         & 1.06.E-01    & - & 1.02.E-02    & - & 1.11.E-01    & - & 1.05.E-02    & - & 1.47.E-01    & - & 1.11.E-02    & - \\
	1/20         & 2.92.E-02    & 1.86 & 1.24.E-03    & 3.03 & 3.00.E-02    & 1.88 & 1.26.E-03    & 3.05 & 3.21.E-02    & 2.19 & 1.30.E-03    & 3.10 \\
	1/40         & 6.31.E-03    & 2.21 & 1.49.E-04    & 3.06 & 6.42.E-03    & 2.22 & 1.51.E-04    & 3.06 & 6.60.E-03    & 2.28 & 1.54.E-04    & 3.07 \\
	1/80         & 1.52.E-03    & 2.05 & 1.83.E-05    & 3.03 & 1.55.E-03    & 2.05 & 1.85.E-05    & 3.03 & 1.57.E-03    & 2.07 & 1.88.E-05    & 3.03 \\
	1/160        & 3.85.E-04    & 1.99 & 2.26.E-06    & 3.01 & 3.90.E-04    & 1.99 & 2.29.E-06    & 3.01 & 3.98.E-04    & 1.98 & 2.32.E-06    & 3.02 \\
	1/320        & 9.25.E-05    & 2.06 & 2.82.E-07    & 3.01 & 9.36.E-05    & 2.06 & 2.85.E-07    & 3.01 & 9.51.E-05    & 2.07 & 2.92.E-07    & 2.99 \\
	1/640        & 2.30.E-05    & 2.01 & 3.92.E-08    & 2.85 & 2.33.E-05    & 2.01 & 3.55.E-08    & 3.00 & 2.36.E-05    & 2.01 & 3.69.E-08    & 2.98 \\
	\bottomrule
	\bottomrule
	\end{tabular}}
	\caption{Maximum errors $E^x_{\max}$ and $E^t_{\max}$ with convergence rates of approximate solutions obtained by the fast predictor-corrector scheme for the advection diffusion equation \eqref{eq:AD_cor} and \eqref{eq:AD_pre} in Example \ref{ex4}. $h$ is determined by $\tau$ with $h=\tau^{2/3}$ and $\tau$ is determined by $h$ with $\tau=h^{3/2}$ when $E^x_{\max}$ and $E^t_{\max}$ are estimated, respectively. }\label{TABLE-ex4-1}
\end{table}

\begin{table}[htp]
	\centering
	\resizebox{1\textwidth}{!}{
	\begin{tabular}{l|llr|llrV{2}l|llr|llr}
		\toprule
		\toprule
		&\multicolumn{3}{c|}{C-PC-Q}&\multicolumn{3}{cV{2}}{F-PC-Q}&&\multicolumn{3}{c|}{C-PC-Q}&\multicolumn{3}{c}{F-PC-Q}     \\
		\cline{2-7}\cline{9-14}
		& $E_{\max}^t$ & roc  & time   & $E_{\max}^t$ & roc  & time   &        & $E_{\max}^x$ & roc    & time   & $E_{\max}^x$ & roc  & time \\
		\midrule
		$h$    &\multicolumn{3}{l}{$\alpha=0.2$}  &\multicolumn{3}{rV{2}}{$\tau=1/80000$}       & $\tau$ & \multicolumn{3}{l}{$\alpha=0.2$} &\multicolumn{3}{r}{$h=1/10000$}       \\
		\midrule
1/10  & 4.01.E-03 & - & 0.37  & 4.01.E-03 & - & 0.29  & 1/10  & 9.11.E-02 & - & 160.51 & 9.11.E-02 & - & 0.43 \\
1/20  & 4.48.E-04 & 3.16 & 0.63  & 4.48.E-04 & 3.16 & 0.51  & 1/20  & 2.21.E-02 & 2.04 & 160.01 & 2.21.E-02 & 2.04 & 0.47 \\
1/40  & 5.28.E-05 & 3.08 & 1.36  & 5.28.E-05 & 3.08 & 0.92  & 1/40  & 5.50.E-03 & 2.01 & 169.92 & 5.50.E-03 & 2.01 & 0.58 \\
1/80  & 6.38.E-06 & 3.05 & 3.63  & 6.38.E-06 & 3.05 & 1.85  & 1/80  & 1.38.E-03 & 1.99 & 211.10 & 1.38.E-03 & 1.99 & 0.80 \\
1/160 & 7.82.E-07 & 3.03 & 9.79  & 7.82.E-07 & 3.03 & 4.02  & 1/160 & 3.58.E-04 & 1.95 & 226.09 & 3.58.E-04 & 1.95 & 1.20 \\
1/320 & 1.03.E-07 & 2.93 & 25.37 & 1.03.E-07 & 2.93 & 8.56  & 1/320 & 1.02.E-04 & 1.81 & 266.34 & 1.02.E-04 & 1.81 & 1.56 \\
1/640 & 2.16.E-08 & 2.25 & 91.31 & 2.15.E-08 & 2.25 & 12.88 & 1/640 & 4.11.E-05 & 1.32 & 304.85 & 4.11.E-05 & 1.32 & 2.29 \\
		\midrule
		$h$    &\multicolumn{6}{lV{2}}{$\alpha=0.5$}        & $\tau$ & \multicolumn{6}{l}{$\alpha=0.5$}        \\
		\midrule
1/10  & 4.15.E-03 & - & 0.36   & 4.19.E-03 & - & 0.26  & 1/10  & 9.19.E-02 & - & 159.25 & 9.19.E-02 & - & 0.42 \\
1/20  & 4.62.E-04 & 3.17 & 0.60   & 4.68.E-04 & 3.16 & 0.47  & 1/20  & 2.23.E-02 & 2.04 & 161.46 & 2.23.E-02 & 2.04 & 0.46 \\
1/40  & 5.41.E-05 & 3.09 & 1.29   & 5.50.E-05 & 3.09 & 0.93  & 1/40  & 5.53.E-03 & 2.01 & 168.48 & 5.53.E-03 & 2.01 & 0.57 \\
1/80  & 6.54.E-06 & 3.05 & 3.67   & 6.65.E-06 & 3.05 & 1.73  & 1/80  & 1.38.E-03 & 2.00 & 175.62 & 1.38.E-03 & 2.00 & 0.78 \\
1/160 & 8.06.E-07 & 3.02 & 9.64   & 8.20.E-07 & 3.02 & 3.27  & 1/160 & 3.45.E-04 & 2.00 & 189.22 & 3.45.E-04 & 2.00 & 1.20 \\
1/320 & 9.95.E-08 & 3.02 & 43.32  & 1.01.E-07 & 3.02 & 5.99  & 1/320 & 8.66.E-05 & 1.99 & 231.99 & 8.66.E-05 & 1.99 & 1.82 \\
1/640 & 1.83.E-08 & 2.44 & 125.35 & 1.85.E-08 & 2.45 & 12.06 & 1/640 & 2.20.E-05 & 1.97 & 309.03 & 2.20.E-05 & 1.97 & 2.65 \\
		\midrule
		$h$    &\multicolumn{6}{lV{2}}{$\alpha=0.8$}        & $\tau$ & \multicolumn{6}{l}{$\alpha=0.8$}        \\
		\midrule
1/10  & 3.93.E-03 & - & 0.59   & 4.72.E-03 & - & 0.31  & 1/10  & 9.14.E-02 & - & 168.18 & 9.14.E-02 & - & 0.51 \\
1/20  & 4.76.E-04 & 3.05 & 0.96   & 5.33.E-04 & 3.15 & 0.62  & 1/20  & 2.22.E-02 & 2.04 & 173.55 & 2.22.E-02 & 2.04 & 0.54 \\
1/40  & 7.05.E-05 & 2.75 & 1.86   & 6.31.E-05 & 3.08 & 1.19  & 1/40  & 5.50.E-03 & 2.01 & 174.79 & 5.50.E-03 & 2.01 & 0.66 \\
1/80  & 9.58.E-06 & 2.88 & 4.90   & 8.14.E-06 & 2.95 & 2.33  & 1/80  & 1.37.E-03 & 2.00 & 187.53 & 1.37.E-03 & 2.00 & 0.93 \\
1/160 & 1.25.E-06 & 2.94 & 12.34  & 1.05.E-06 & 2.96 & 4.66  & 1/160 & 3.43.E-04 & 2.00 & 206.76 & 3.43.E-04 & 2.00 & 1.42 \\
1/320 & 1.58.E-07 & 2.98 & 30.09  & 1.33.E-07 & 2.97 & 9.22  & 1/320 & 8.56.E-05 & 2.00 & 241.23 & 8.56.E-05 & 2.00 & 2.13 \\
1/640 & 2.11.E-08 & 2.91 & 116.92 & 1.75.E-08 & 2.93 & 18.22 & 1/640 & 2.14.E-05 & 2.00 & 320.79 & 2.14.E-05 & 2.00 & 3.34\\
		\bottomrule
		\bottomrule
	\end{tabular}}
	\caption{Maximum errors $E^t_{\max}$ and CPU times executed C-PC-Q and F-PC-Q for various values $\alpha$ with $\tau=1/80000$ fixed. And maximum errors $E^x_{\max}$ and CPU times executed C-PC-Q and F-PC-Q for various values $\alpha$ with $h=1/10000$ fixed in Example \ref{ex4}}\label{TABLE-ex4-2}
\end{table}

\begin{table}[htp]
	\centering
	\resizebox{1\textwidth}{!}{
		\begin{tabular}{l|llll|llll|llll}
			\toprule
			\toprule
			& $E_{\max}^x$ & roc  & $E_{\max}^t$ & roc  & $E_{\max}^x$ & roc  & $E_{\max}^t$ & roc  & $E_{\max}^x$ & roc  & $E_{\max}^t$ & roc  \\
			\cline{2-13}
			&  \multicolumn{4}{c|}{$\alpha=0.2$}       & \multicolumn{4}{c|}{$\alpha=0.5$}       & \multicolumn{4}{c}{$\alpha=0.8$}       \\
			\midrule
			$\tau,h$     & \multicolumn{12}{c}{F-PC-Q}  \\
			\midrule
1/10  & 2.82.E-03 & - & 2.76.E-04 & - & 2.99.E-03 & - & 2.98.E-04 & - & 3.21.E-03 & 0.00 & 3.72.E-04 & 0.00 \\
1/20  & 7.06.E-04 & 2.00 & 3.58.E-05 & 2.95 & 7.56.E-04 & 1.98 & 3.89.E-05 & 2.94 & 8.75.E-04 & 1.87 & 5.44.E-05 & 2.77 \\
1/40  & 1.77.E-04 & 2.00 & 4.43.E-06 & 3.01 & 1.91.E-04 & 1.98 & 4.84.E-06 & 3.01 & 2.47.E-04 & 1.82 & 7.25.E-06 & 2.91 \\
1/80  & 4.43.E-05 & 2.00 & 5.53.E-07 & 3.00 & 4.81.E-05 & 1.99 & 6.05.E-07 & 3.00 & 6.69.E-05 & 1.89 & 9.42.E-07 & 2.94 \\
1/160 & 1.11.E-05 & 2.00 & 6.93.E-08 & 3.00 & 1.21.E-05 & 1.99 & 7.59.E-08 & 3.00 & 1.76.E-05 & 1.93 & 1.20.E-07 & 2.97 \\
1/320 & 2.77.E-06 & 2.00 & 8.69.E-09 & 3.00 & 3.03.E-06 & 2.00 & 9.50.E-09 & 3.00 & 4.58.E-06 & 1.94 & 1.52.E-08 & 2.98 \\
1/640 & 6.93.E-07 & 2.00 & 5.04.E-09 & 0.78 & 7.58.E-07 & 2.00 & 1.25.E-09 & 2.93 & 1.17.E-06 & 1.96 & 2.29.E-09 & 2.74\\
			\bottomrule
			\bottomrule
	\end{tabular}}
	\caption{Maximum errors $E^x_{\max}$ and $E^t_{\max}$ with convergence rates of approximate solutions obtained by the fast predictor-corrector scheme for the diffusion equation \eqref{eq:D_cor} and \eqref{eq:D_pre} in Example \ref{ex5}. $h$ is determined by $\tau$ with $h=\tau^{2/3}$ and $\tau$ is determined by $h$ with $\tau=h^{3/2}$ when $E^x_{\max}$ and $E^t_{\max}$ are estimated, respectively.}\label{TABLE-ex5-1}
\end{table}

\begin{table}[htp]
	\centering
	\resizebox{1\textwidth}{!}{
		\begin{tabular}{l|llr|llrV{2}l|llr|llr}
			\toprule
			\toprule
			&\multicolumn{3}{c|}{C-PC-Q}&\multicolumn{3}{cV{2}}{F-PC-Q}&&\multicolumn{3}{c|}{C-PC-Q}&\multicolumn{3}{c}{F-PC-Q}     \\
			\cline{2-7}\cline{9-14}
			& $E_{\max}^t$ & roc  & time   & $E_{\max}^t$ & roc  & time   &        & $E_{\max}^x$ & roc    & time   & $E_{\max}^x$ & roc  & time \\
			\midrule
			$h$    &\multicolumn{3}{l}{$\alpha=0.2$}  &\multicolumn{3}{rV{2}}{$\tau=1/40000$}       & $\tau$ & \multicolumn{3}{l}{$\alpha=0.2$} &\multicolumn{3}{r}{$h=1/4000$}       \\
			\midrule
1/10  & 6.46.E-06 & -  & 0.16  & 6.46.E-06 & -  & 0.16 & 1/10  & 2.84.E-03 & - & 25.00 & 2.84.E-03 & - & 0.19 \\
1/20  & 9.18.E-07 & 2.81  & 0.31  & 9.18.E-07 & 2.81  & 0.27 & 1/20  & 7.10.E-04 & 2.00 & 25.34 & 7.10.E-04 & 2.00 & 0.20 \\
1/40  & 1.22.E-07 & 2.91  & 0.71  & 1.22.E-07 & 2.91  & 0.54 & 1/40  & 1.77.E-04 & 2.00 & 25.52 & 1.77.E-04 & 2.00 & 0.24 \\
1/80  & 1.53.E-08 & 2.99  & 1.59  & 1.53.E-08 & 2.99  & 0.94 & 1/80  & 4.44.E-05 & 2.00 & 26.74 & 4.44.E-05 & 2.00 & 0.35 \\
1/160 & 2.83.E-09 & 2.43  & 4.55  & 2.83.E-09 & 2.43  & 1.71 & 1/160 & 1.11.E-05 & 2.00 & 32.16 & 1.11.E-05 & 2.00 & 0.55 \\
1/320 & 1.73.E-09 & 0.71  & 14.82 & 1.73.E-09 & 0.71  & 3.41 & 1/320 & 2.78.E-06 & 2.00 & 39.14 & 2.78.E-06 & 2.00 & 0.94 \\
1/640 & 5.08.E-09 & -1.55 & 53.79 & 5.08.E-09 & -1.55 & 6.87 & 1/640 & 6.98.E-07 & 1.99 & 50.72 & 6.98.E-07 & 1.99 & 1.67 \\
			\midrule
			$h$    &\multicolumn{6}{lV{2}}{$\alpha=0.5$}        & $\tau$ & \multicolumn{6}{l}{$\alpha=0.5$}        \\
			\midrule
1/10  & 5.16.E-06 & -  & 0.14  & 5.16.E-06 & -  & 0.16 & 1/10  & 3.12.E-03 & - & 25.63 & 3.12.E-03 & - & 0.19 \\
1/20  & 7.74.E-07 & 2.74  & 0.29  & 7.74.E-07 & 2.74  & 0.26 & 1/20  & 7.78.E-04 & 2.00 & 26.17 & 7.78.E-04 & 2.00 & 0.21 \\
1/40  & 1.05.E-07 & 2.88  & 0.63  & 1.05.E-07 & 2.88  & 0.46 & 1/40  & 1.95.E-04 & 2.00 & 26.20 & 1.95.E-04 & 2.00 & 0.25 \\
1/80  & 1.44.E-08 & 2.87  & 1.51  & 1.44.E-08 & 2.87  & 1.00 & 1/80  & 4.86.E-05 & 2.00 & 30.25 & 4.86.E-05 & 2.00 & 0.35 \\
1/160 & 1.33.E-09 & 3.44  & 4.65  & 1.33.E-09 & 3.44  & 1.78 & 1/160 & 1.22.E-05 & 2.00 & 30.86 & 1.22.E-05 & 2.00 & 0.52 \\
1/320 & 1.39.E-09 & -0.06 & 15.02 & 1.39.E-09 & -0.06 & 3.49 & 1/320 & 3.04.E-06 & 2.00 & 38.23 & 3.04.E-06 & 2.00 & 0.92 \\
1/640 & 1.50.E-09 & -0.11 & 53.26 & 1.50.E-09 & -0.11 & 7.39 & 1/640 & 7.60.E-07 & 2.00 & 51.08 & 7.60.E-07 & 2.00 & 1.68 \\
			\midrule
			$h$    &\multicolumn{6}{lV{2}}{$\alpha=0.8$}        & $\tau$ & \multicolumn{6}{l}{$\alpha=0.8$}        \\
			\midrule
1/10  & 2.84.E-05 & -  & 0.23  & 2.84.E-05 & -  & 0.18 & 1/10  & 5.06.E-03 & - & 25.76 & 5.06.E-03 & - & 0.24 \\
1/20  & 3.78.E-06 & 2.91  & 0.41  & 3.78.E-06 & 2.91  & 0.35 & 1/20  & 1.26.E-03 & 2.01 & 26.92 & 1.26.E-03 & 2.01 & 0.23 \\
1/40  & 4.85.E-07 & 2.96  & 0.82  & 4.85.E-07 & 2.96  & 0.64 & 1/40  & 3.14.E-04 & 2.00 & 26.16 & 3.14.E-04 & 2.00 & 0.26 \\
1/80  & 6.19.E-08 & 2.97  & 2.00  & 6.19.E-08 & 2.97  & 1.27 & 1/80  & 7.86.E-05 & 2.00 & 27.99 & 7.86.E-05 & 2.00 & 0.37 \\
1/160 & 7.96.E-09 & 2.96  & 5.04  & 7.96.E-09 & 2.96  & 2.49 & 1/160 & 1.96.E-05 & 2.00 & 30.38 & 1.96.E-05 & 2.00 & 0.62 \\
1/320 & 2.16.E-09 & 1.88  & 17.31 & 2.16.E-09 & 1.88  & 4.96 & 1/320 & 4.91.E-06 & 2.00 & 37.60 & 4.91.E-06 & 2.00 & 1.04 \\
1/640 & 3.05.E-09 & -0.50 & 58.40 & 3.05.E-09 & -0.50 & 9.75 & 1/640 & 1.23.E-06 & 2.00 & 51.48 & 1.23.E-06 & 2.00 & 1.95\\
			\bottomrule
			\bottomrule
	\end{tabular}}
	\caption{Maximum errors $E^t_{\max}$ and CPU times executed C-PC-Q and F-PC-Q for various values $\alpha$ with $\tau=1/40000$ fixed. And maximum errors $E^x_{\max}$ and CPU times executed C-PC-Q and F-PC-Q for various values $\alpha$ with $h=1/4000$ fixed in Example \ref{ex5}}\label{TABLE-ex5-2}
\end{table}

\begin{figure}[htp]
	\begin{center}
		\begin{tabular}{c}
			\resizebox{110mm}{!}{\includegraphics{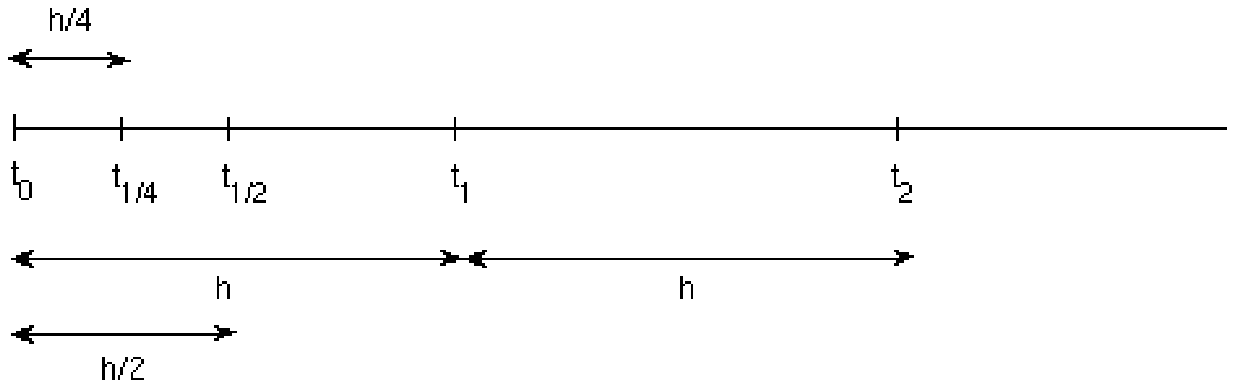}}
		\end{tabular}
		\caption{\small{Grid of the linear-quadratic start-up algorithm.}}
		\label{fig_startup_quad}
	\end{center}
\end{figure}

\section{Conclusion}\label{sec_con}
In this paper, new predictor-corrector methods using linear and quadratic interpolations for nonlinear Caputo-Fabrizio fractional differential equations were introduced to deal with the integral equation with non-singular kernel \eqref{model} which is equivalent to the model problem \eqref{intro_mod}. The fast predictor-corrector schemes using the recurrence relation of the memory term were also introduced. The edges what the proposed methods have on that
\begin{enumerate}
	\item the global orders of convergence are $\mathcal{O}(h^2)$ and $\mathcal{O}(h^3)$ for linear and quadratic interpolations, respectively for the fractional order $0<\alpha<1$.
	\item The fast predictor-corrector algorithm requires approximately $\mathcal{O}(N)$ arithmetic operations while the regular predictor-corrector schemes requires approximately $\mathcal{O}(N^2)$ due to the memory term. Thus the fast algorithm has low computational complexity so that it reduces the computational cost comparing with the regular predictor-corrector schemes.
\end{enumerate}
The advantages above were numerically justified by two nonlinear fractional ordinary differential equations and two time-fractional partial differential equations. The performance of the proposed methods with the efficiency of the fast algorithm was demonstrated in Example \ref{ex2} and \ref{ex3}. We showed that the proposed methods can be implemented to solve the diffusion equations \ref{ex4} and \ref{ex5} withouth loss of accuracy and global error estimates proved theoretically and without compromising the low computational cost of the fast algorithm.


\section*{Acknowledgments}
This work was supported by Basic Science Research Program through the National Research Foundation of Korea (NRF) funded by the Ministry of Education (NRF-2014R1A1A2A16051147). We would like to address our thanks to professor Changpin Li for the correction of the typos of the \emph{Alg II} scheme in \cite{LCY}.

\appendix
\section{Start-up of the Scheme}\label{appendix-start-up}
 To find a desired accuracy for ${y}_{1}$ and $y_2$, we find the approximate solutions at $t_{1/4}$ and $y_{1/2}$ using the linear-quadratic predictor-corrector scheme described below. The grids are shown in Figure \ref{fig_startup_quad}.
 
 
First, let us define the following forms to be used frequently.
 \begin{align*}
	&\hat{A}^{c,d,e}_{n,a,b}=\int_{t_a}^{t_b} \left(\frac{\tau-t_d}{t_c-t_d}\right)\left( \frac{\tau-t_e}{t_c-t_e}
	\right)  \exp[-\beta(t_n-\tau)] d\tau \\
	&\hat{B}^{c,d}_{n,a,b}=\int_{t_a}^{t_b}\left(\frac{\tau-t_d}{t_c-t_d} \right) \exp[-\beta(t_n-\tau)]d\tau.
\end{align*}
The detail algorithm is described as follows
\begin{alg}\label{quad_alg} (Start-up Procedure)
\begin{enumerate}
\item Approximate $\tilde{y}_{1/4}$:
\begin{itemize}
\item[-] Predict $\tilde{y}^P_{1/4}$: 
\begin{align}\nonumber
[1-\beta\hat{B}_{1/4,0,1/4}^{1/4,0}]\tilde{y}^P_{1/4}=\tilde{y}_0e^{-\beta t_{1/4}}+\frac{1-\alp}{M(\alp)}f(t_{1/4},\tilde{y}_0)+\beta[\hat{B}_{1/4,0,1/4}^{0,1/4}\tilde{y}_0]
\end{align}
\item[-] Correct $\tilde{y}_{1/4}$: 
\begin{align}\nonumber
[1-\beta\hat{B}_{1/4,0,1/4}^{1/4,0}]\tilde{y}_{1/4}=\tilde{y}_0e^{-\beta t_{1/4}}+\frac{1-\alp}{M(\alp)}\tilde{f}_{1/4}^P+\beta[\hat{B}_{1/4,0,1/4}^{0,1/4}\tilde{y}_0]
\end{align}
\end{itemize}
\item  Approximate $\tilde{y}_{1/2}$:
\begin{itemize}
	\item[-] Predict $\tilde{y}^P_{1/2}$:  
	\begin{align}\nonumber
	[1-\beta\hat{A}_{1/2,0,1/2}^{1/2,0,1/4}]\tilde{y}^P_{1/2}=\tilde{y}_0e^{-\beta t_{1/2}}+\frac{1-\alp}{M(\alp)}[2f_{1/4}-f_0]+\beta[\hat{A}_{1/2,0,1/2}^{0,1/4,1/2}\tilde{y}_{0}+\tilde{A}_{1/2,0,1/2}^{1/4,0,1/2}\tilde{y}_{1/4}]
	\end{align}
	\item[-] Correct $\tilde{y}_{1/2}$: 
	\begin{align}\nonumber
	[1-\beta\hat{A}_{1/2,0,1/2}^{1/2,0,1/4}]\tilde{y}_{1/2}=\tilde{y}_0e^{-\beta t_{1/2}}+\frac{1-\alp}{M(\alp)}\tilde{f}_{1/2}^P+\beta[\hat{A}_{1/2,0,1/2}^{0,1/4,1/2}\tilde{y}_{0}+\tilde{A}_{1/2,0,1/2}^{1/4,0,1/2}\tilde{y}_{1/4}]
	\end{align}
\end{itemize}
\item Approximate $\tilde{y}_1$:
\begin{itemize}
	\item[-] Predict $\tilde{y}^P_{1}$:  
	\begin{align}\nonumber
	[1-\beta\hat{A}_{1,0,1}^{0,1/2,1}]\tilde{y}^P_{1}=\tilde{y}_0e^{-\beta t_{1}}+\frac{1-\alp}{M(\alp)}[6f_{1/2}-8f_{1/4}-3f_0]+\beta[\hat{A}_{1,0,1}^{0,1/2,1}\tilde{y}_{0}+\hat{A}_{1,0,1}^{1/2,0,1}\tilde{y}_{1/2}]
	\end{align}
	\item[-] Correct $\tilde{y}_{1}$: 
	\begin{align}\nonumber
	[1-\beta\hat{A}_{1,0,1}^{0,1/2,1}]\tilde{y}_{1}=\tilde{y}_0e^{-\beta t_{1}}+\frac{1-\alp}{M(\alp)}\tilde{f}_{1}^P+\beta[\hat{A}_{1,0,1}^{0,1/2,1}\tilde{y}_{0}+\hat{A}_{1,0,1}^{1/2,0,1}\tilde{y}_{1/2}]
	\end{align}
\end{itemize}

\item Approximate $\tilde{y}_2$:
\begin{itemize}
	\item[-] Predict $\tilde{y}^P_{2}$:  
	\begin{align*}
	[1-\beta\hat{A}_{2,1,2}^{0,1,2}]\tilde{y}^P_{2}=&\tilde{y}_0e^{-\beta t_{2}}+\frac{1-\alp}{M(\alp)}[6f_{1}-8f_{1/2}-3f_0]\\		      
	+&\beta[\hat{A}_{2,0,1}^{0,1/2,1}\tilde{y}_{0}+\hat{A}_{2,0,1}^{1/2,0,1}\tilde{y}_{1/2}+\hat{A}_{2,0,1}^{1,0,1/2}\tilde{y}_1
	+\hat{A}_{2,1,2}^{0,1,2}\tilde{y}_0+\hat{A}_{2,1,2}^{1,0,2}\tilde{y}_1]
	\end{align*}
	\item[-] Correct $\tilde{y}_{2}$: 
	\begin{align*}
[1-\beta\hat{A}_{2,1,2}^{0,1,2}]\tilde{y}_{2}=&\tilde{y}_0e^{-\beta t_{2}}+\frac{1-\alp}{M(\alp)}f_{2}^P\\
           +&\beta[\hat{A}_{2,0,1}^{0,1/2,1}\tilde{y}_{0}+\hat{A}_{2,0,1}^{1/2,0,1}\tilde{y}_{1/2}+\hat{A}_{2,0,1}^{1,0,1/2}\tilde{y}_1+\hat{A}_{2,1,2}^{0,1,2}\tilde{y}_0+\hat{A}_{2,1,2}^{1,0,2}\tilde{y}_1]
	\end{align*}
\end{itemize}

\end{enumerate}
\end{alg}

\bibliography{bibtex_CF}

\begin{thebibliography}{10}
\expandafter\ifx\csname url\endcsname\relax
  \def\url#1{\texttt{#1}}\fi
\expandafter\ifx\csname urlprefix\endcsname\relax\def\urlprefix{URL }\fi
\expandafter\ifx\csname href\endcsname\relax
  \def\href#1#2{#2} \def\path#1{#1}\fi

\bibitem{GHPS}
I.~Goychuk, E.~Heinsalu, M.~Patriarca, G.~Schmid, P.~H{\"a}nggi, Current and
  universal scaling in anomalous transport, Physical Review E 73~(2) (2006)
  020101.

\bibitem{KRS}
R.~Klages, G.~Radons, I.~M. Sokolov, Anomalous transport: foundations and
  applications, John Wiley \& Sons, 2008.

\bibitem{Z}
G.~M. Zaslavsky, Chaos, fractional kinetics, and anomalous transport, Physics
  reports 371~(6) (2002) 461--580.

\bibitem{BWM1}
D.~A. Benson, S.~W. Wheatcraft, M.~M. Meerschaert, Application of a fractional
  advection-dispersion equation, Water resources research 36~(6) (2000)
  1403--1412.

\bibitem{BC}
R.~L. Bagley, R.~Calico, Fractional order state equations for the control of
  viscoelastically damped structures, Journal of Guidance, Control, and
  Dynamics 14~(2) (1991) 304--311.

\bibitem{Ma}
F.~Mainardi, Fractional calculus and waves in linear viscoelasticity: an
  introduction to mathematical models, World Scientific, 2010.

\bibitem{petravs2011fractional}
I.~Petr{\'a}{\v{s}}, Fractional-order nonlinear systems: modeling, analysis and
  simulation, Springer Science \& Business Media, 2011.

\bibitem{baleanu2011fractional}
D.~Baleanu, J.~A.~T. Machado, A.~C. Luo, Fractional dynamics and control,
  Springer Science \& Business Media, 2011.

\bibitem{P}
I.~Podlubny, Fractional differential equations: an introduction to fractional
  derivatives, fractional differential equations, to methods of their solution
  and some of their applications, Vol. 198, Elsevier, 1998.

\bibitem{atangana2017fractional}
A.~Atangana, Fractional operators with constant and variable order with
  application to geo-hydrology, Academic Press, 2017.

\bibitem{SAM}
J.~Sabatier, O.~P. Agrawal, J.~T. Machado, Advances in fractional calculus,
  Vol.~4, Springer, 2007.

\bibitem{AG}
A.~Atangana, J.~G{\'o}mez-Aguilar, A new derivative with normal distribution
  kernel: Theory, methods and applications, Physica A: Statistical mechanics
  and its applications 476 (2017) 1--14.

\bibitem{CF1}
M.~Caputo, M.~Fabrizio, A new definition of fractional derivative without
  singular kernel, Progr. Fract. Differ. Appl 1~(2) (2015) 73--85.

\bibitem{atangana2017}
K.~M. Owolabi, A.~Atangana, Analysis and application of new fractional
  adams-bashforth scheme with caputo-fabrizio derivative, Chaos, Solitons and
  Fractals 105 (2017) 111--119.

\bibitem{DFF02}
K.~Diethelm, N.~J. Ford, A.~D. Freed, A predictor-corrector approach for the
  numerical solution of fractional differential equations, Nonlinear Dynamics
  29~(1-4) (2002) 3--22.

\bibitem{JANG1}
T.~Nguyen, B.~Jang, A high-order predictor-corrector method for solving
  nonlinear differential equations of fractional order, Fractional Calculus and
  applied analysis 20~(2) (2017) 447--476.

\bibitem{Brunner-linear}
H.~Brunner, A.~Pedas, G.~Vainikko, Piecewise polynomial collocation methods for
  linear {V}olterra integro-differential equations with weakly singular
  kernels, SIAM J. Numer. Anal. 39 (2001) 957--982.

\bibitem{Brunner-nonlinear}
H.~Brunner, A.~Pedas, G.~Vainikko, The piecewise polynomial collocation method
  for nonlinear weakly singular {V}olterra equations, Mathematics of
  Computation 68~(227) (1999) 1079--1095.

\bibitem{Chen1}
T.~T. Y.~Chen, Convergence analysis of the {J}acobi spectral-collocation
  methods for {V}olterra integral equation with a weakly singular kernel, Appl.
  Math. Modell. 79 (2010) 147--167.

\bibitem{Gu}
Z.~Gu, Piecewise spectral collocation method for system of {V}olterra integral
  equations, Adv. Comput. Math. 43 (2017) 385--409.

\bibitem{Gu-Chen}
Z.~Gu, Y.~Chen, Piecewise {L}egendre spectral-collocation method for {V}olterra
  integro-differential equations, LMS J. Comput. Math. 18 (2015) 231--249.

\bibitem{Huang1}
M.~S. C.~Huang, A spectral collocation method for a weakly singular {V}olterra
  integral equation of the second kind, Adv. Comput. Math. 42 (2016)
  1015--1030.

\bibitem{Shen1}
J.~Shen, C.~Sheng, Z.~Wang, Generalized {J}acobi spectral-{G}alerkin method for
  nonlinear {V}olterra integral equations with weakly singular kernels, J.
  Math. Study 48~(4) (2015) 315--329.

\bibitem{Wei1}
Y.~C. Y.~Wei, Convergence analysis of the spectral methods for weakly singular
  {V}olterra integro-differential equations with smooth solutions, Adv. Appl.
  Math. Mech. 4(1) (2012) 1--20.

\bibitem{Xianjuan1}
T.~T. X.~Li, Convergence analysis of {J}acobi spectral collocation methods for
  {A}bel-{V}olterra integral equations of second kind, Front. Math. China 7(1)
  (2012) 69--84.

\bibitem{KOCA2018278}
I.~Koca, Efficient numerical approach for solving fractional partial
  differential equations with non-singular kernel derivatives, Chaos, Solitons
  \& Fractals 116 (2018) 278 -- 286.

\bibitem{atangana2016numerical}
A.~Atangana, R.~T. Alqahtani, Numerical approximation of the space-time
  caputo-fabrizio fractional derivative and application to groundwater
  pollution equation, Advances in Difference Equations 2016~(1) (2016) 156.

\bibitem{DK}
J.~Dixon, S.~McKee, Weakly singular discrete gronwall inequalities,
  ZAMM-Journal of Applied Mathematics and Mechanics/Zeitschrift f{\"u}r
  Angewandte Mathematik und Mechanik 66~(11) (1986) 535--544.

\bibitem{CX}
J.~Cao, C.~Xu, A high order schema for the numerical solution of the fractional
  ordinary differential equations, Journal of Computational Physics 238 (2013)
  154--168.

\bibitem{D}
K.~Diethelm, The analysis of fractional differential equations: An
  application-oriented exposition using differential operators of Caputo type,
  Springer Science \& Business Media, 2010.

\bibitem{LCY}
C.~Li, A.~Chen, J.~Ye, Numerical approaches to fractional calculus and
  fractional ordinary differential equation, Journal of Computational Physics
  230~(9) (2011) 3352--3368.

\end{thebibliography}

\end{document}